\documentclass[11 pt]{amsart}

\usepackage{times}
\usepackage{geometry}
\usepackage{amssymb}
\usepackage{latexsym, amsmath, amscd,amsthm}
\usepackage{graphicx}
\usepackage[percent]{overpic}
\usepackage{pdfsync}
\usepackage{units}
\usepackage{hyperref}
\usepackage{diagrams}
\usepackage{paralist}

\newtheorem{theorem}{Theorem}

\newtheorem{lemma}[theorem]{Lemma}
\newtheorem{proposition}[theorem]{Proposition}
\newtheorem{corollary}[theorem]{Corollary}

\theoremstyle{definition}
\newtheorem{definition}[theorem]{Definition}

\newtheorem{remark}[theorem]{Remark}

\def\defn#1{Definition~\ref{def:#1}}
\def\thm#1{Theorem~\ref{thm:#1}}
\def\lem#1{Lemma~\ref{lem:#1}}

\def\prop#1{Proposition~\ref{prop:#1}}
\def\cor#1{Corollary~\ref{cor:#1}}

\def\rmark#1{Remark~\ref{rmark:#1}}

\newcommand{\R}{\mathbb{R}}
\newcommand{\Z}{\mathbb{Z}}
\newcommand{\bdry}{\partial}
\newcommand{\bdy}{\partial}
\newcommand{\transverse}{\pitchfork}
\newcommand{\Span}{\operatorname{Span}}
\newcommand{\cross}{\times}
\def\co{\colon\!}

\newcommand{\pij}{\pi_{ij}}

\newcommand\norm[1]{\left| #1 \right|}
\newcommand{\p}{\protect\overrightarrow{\mathbf{p}} }
\newcommand{\q}{\protect\overrightarrow{\mathbf{q}} }
\newcommand{\h}{\protect\overrightarrow{\mathbf{h}} }
\newcommand{\cv}{\protect\overrightarrow{\mathbf{v}} }
\newcommand{\x}{\protect\overrightarrow{\mathbf{x}} }
\newcommand{\bfh}{\mathbf{h}}
\newcommand{\bfm}{\mathbf{m}}
\newcommand{\bfp}{\mathbf{p}}
\newcommand{\bfq}{\mathbf{q}}
\newcommand{\bfr}{\mathbf{r}}
\newcommand{\bfv}{\mathbf{v}}
\newcommand{\bfw}{\mathbf{w}}
\newcommand{\bfx}{\mathbf{x}}
\newcommand{\bfy}{\mathbf{y}}
\newcommand{\bfzero}{\mathbf{0}}
\newcommand{\cnm}{C_n[M]}
\newcommand{\cnmo}{C_n(M)}
\newcommand{\cnmno}{C_n(M\times N)}
\newcommand{\cnn}{C_n[N]}
\newcommand{\cnno}{C_n(N)}
\newcommand{\cnr}{C_n[\R^k]}
\newcommand{\cnro}{C_n(\R^k)}
\newcommand{\cfr}{C_4[\R^k]}
\newcommand{\cfro}{C_4(\R^k)}
\newcommand{\qcfr}{\hat{C}_4[\R^k]}
\newcommand{\qcfro}{\hat{C}_4(\R^k)}
\newcommand{\cfs}{C_4[S^1]}
\newcommand{\cns}{C_n[S^1]}
\newcommand{\cnso}{C_n(S^1)}
\newcommand{\cnsl}{C_n[S^l]}
\newcommand{\cnslo}{C_n(S^l)}
\newcommand{\cfg}{C_4[\gamma]}
\newcommand{\cfog}{C^0_4[\gamma]}
\newcommand{\qcfog}{\hat{C}^0_4[\gamma]}
\newcommand{\cng}{C_n[\gamma]}
\newcommand{\cnf}{C_n[f]}
\newcommand{\slq}{Slq}
\newcommand{\qslq}{\widehat{S}lq}
\newcommand{\paren}{parenthesization }
\newcommand{\parens}{parenthesizations }
\newcommand{\FTCWC}{\operatorname{FTCWC}}
\def\mapright#1#2{
\mathop{\longrightarrow}\limits^{\scriptstyle#1}_{\scriptstyle#2}}
\def\vic#1{\mathbf{#1}}
\graphicspath{./figs}

\begin{document}
\title[]{Configuration Spaces, Multijet Transversality, and the Square-Peg Problem}
\author{Jason Cantarella}
\address{Mathematics Department, University of Georgia, Athens GA 30602}
\email{jason@math.uga.edu} 
\author{Elizabeth Denne}
\address{Mathematics Department, Washington \& Lee University, Lexington VA 24450}
\email{dennee@wlu.edu}
\author{John McCleary}
\address{Mathematics \& Statistics Department, Vassar College, Poughkeepsie NY 12604}
\email{mccleary@vassar.edu}

\makeatletter								
\@namedef{subjclassname@2020}{%
  \textup{2020} Mathematics Subject Classification}
\makeatother


\begin{abstract} 
We prove a transversality ``lifting property'' for compactified configuration spaces as an application of the multijet transversality theorem: given a submanifold of configurations of points on an embedding of a compact manifold $M$ in Euclidean space, we can find a dense set of smooth embeddings of $M$ for which the corresponding configuration space of points is transverse to any submanifold of the configuration space of points in Euclidean space, as long as the two submanifolds of compactified configuration space are boundary-disjoint. 
We use this setup to provide an attractive proof of the square-peg problem: there is a dense family of smoothly embedded circles in the plane where each simple closed curve has an odd number of inscribed squares, and there is a dense family of smoothly embedded circles in $\R^n$ where each simple closed curve has an odd number of inscribed square-like quadrilaterals.
\end{abstract}

\date{\today}
\maketitle


\section{Introduction}\label{sect:intro} 

Given a simple closed curve (a Jordan curve) $\gamma$ in $\R^2$, can we find four points on $\gamma$ that form a square? This question was posed by O.~Toeplitz in 1911 \cite{Toeplitz},
and it has drawn the attention of many mathematicians since that time. Thinking of the Jordan curve as a ``round hole'', the problem has been affectionately dubbed the ``square-peg'' problem. We say that the square is \emph{inscribed} in $\gamma$ when the vertices lie on the curve. (We do not require that the square lie entirely in the interior of the curve.) In the form posed by Toeplitz, the problem remains open.  
The earliest contribution to the problem is due to A.~Emch \cite{MR1506193, MR1506239,MR1506274} who showed that there are squares on convex curves. Progress on the square-peg problem has chiefly been extension of the regularity class of simple closed curves for which the square can be found. The interested reader can find numerous articles \cite{FTCWC, MR1133201, Pak-Discrete-Poly-Geom,  MR3184501, MR3731730} summarizing the problem, and describing the classes of curves for which the Toeplitz conjecture has been proved.

The square-peg problem can be framed in terms of configuration spaces of points. First, consider the compactified configuration space $C_4[\R^2]$ of 4-tuples of points in the plane as a manifold-with-boundary (and corners). Then, the existence of inscribed squares can be rephrased more simply as a question about the intersections of two submanifolds of $C_4[\R^2]$. The first is the submanifold of 4-tuples of points on an embedding $\gamma\co S^1\hookrightarrow \R^2$ of a circle in the plane, denoted by $C_4[\gamma(S^1)]$; the second is the submanifold of squares in the plane, denoted by $\slq$. 

Thus, the square-peg problem is an example of the general problem of finding special geometric configurations on families of manifolds.  Theorems of a similar nature include S.~Kakutani's theorem \cite{MR7267} that a compact convex body in $\R^3$ has a circumscribed cube,  that is, a cube each of whose faces touch the convex body; the work of A.~Akoypan and R.~Karasev \cite{MR3016979} answers whether a convex  polytope admits an inscribed regular octahedron. In addition, there is the work of P.V.M.~Blagojevi\'c and G.~Ziegler \cite{MR2515779} on inscribed tetrahedra in spheres; G.~Kuperberg \cite{MR1703205} and V.V.~ Makeev \cite{MR2307358} on inscribed and circumscribed polyhedra in convex bodies and spheres. Compactified configuration spaces have also been used by S.T.~Vre\'{c}ica and R.T.~\v{Z}ivaljevi\'{c}, T. Rade  \cite{MR2823976} in their paper looking at the polygonal peg problem (inscribed affine regular hexagons in smooth Jordan curves, and inscribed parallelograms in smooth simple closed curves in $\R^3$).  There have also been many  papers  \cite{MR3810027, MR4061975,  greene2020cyclic, MR4298749, hugelmeyer2018smooth,  MR4298748, MR2038265, matschke2020quadrilaterals, schwartz2018rectangle, MR4142923} examining quadrilaterals inscribed in curves and,  more recently, making progress towards solving the rectangular-peg problem (finding rectangles of any aspect ratio inscribed in Jordan curves).


Here is an outline of our approach. We suppose that we seek a special configuration of $n$ points to be found on a compact manifold $M$ that is smoothly embedded in $\R^k$. We let the smooth embedding be denoted by $\gamma\co M\hookrightarrow \R^k$. Then, in the compactified configuration space $C_n[\R^k]$ of $n$ points in $\R^k$, there are two subspaces of interest:
$C_n[\gamma(M)]$, the compactified configuration space of $n$ points on $\gamma(M)$; and $Z$ the subspace of tuples of $n$ points in $\R^k$ that satisfy the conditions of a special configuration. For example, in the square-peg problem, $\gamma(S^1)$ is a smooth simple closed curve in $\R^2$, and the special configurations $Z$ are squares in $\R^2$.

Why do we use compactified configuration spaces? The open manifold $C_n(\R^k)$ of $n$-tuples of distinct points in $\R^k$ contains $C_n(\gamma(M))$ and the interior of $Z$, and we can ask if these submanifolds intersect. However, intersection theory in open manifolds is difficult. The compactification of configuration spaces of W.~Fulton and R.~MacPherson \cite{MR1259368} as developed by D.P.~Sinha \cite{newkey119}, provides the tools to make intersection theory reasonable. We give an overview of this theory in Section~\ref{sect:config}.

In Section~\ref{sect:transversality}, we look at intersections of submanifolds of configuration spaces and transversality. Suppose there is a different, well known, smooth embedding of $M$ in $\R^k$ (via $i\co M\hookrightarrow \R^k$), and assume that the configuration space $C_n[i(M)]$ is transverse to $Z$ in $C_n[\R^k]$. Also assume that $i(M)$ is smoothly homotopy equivalent to $\gamma(M)$ in $\R^k$. Standard transversality arguments should allow us to vary $C_n[i(M)]$ to $C_n[\gamma(M)]$ while maintaining the transversality of the intersection with $Z$. This idea is illustrated in Figure~\ref{fig:main-idea}.
The difficulties of this argument include: 
\begin{compactenum}
\item It is possible that special configurations on $\gamma(M)$ shrink away to the boundary
of $C_n[\R^k]$ during the isotopy. We overcome this problem by assuming that the boundaries  $\partial Z$ and $\partial C_n[\gamma(M)]$ are disjoint in $\partial C_n[\R^k]$.
\item In order to apply transversality arguments, we need to be able to perturb $C_n[\gamma(M)]$ so the intersection with $Z$ is transverse. However, when we do so, there is no guarantee that the varied submanifold consists of configurations on a perturbed smooth embedding of $M$ in $\R^k$. We deal with this issue by applying the multijet transversality theorem (\thm{multijet}).
\end{compactenum}

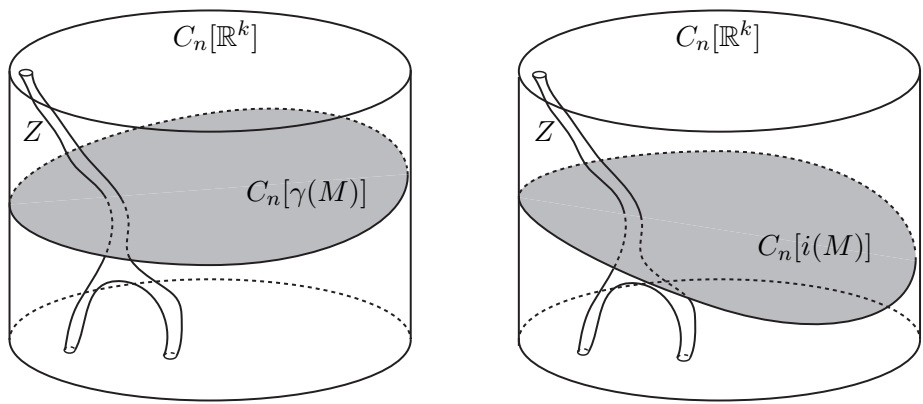
\begin{figure}
\begin{overpic}[scale=0.75]{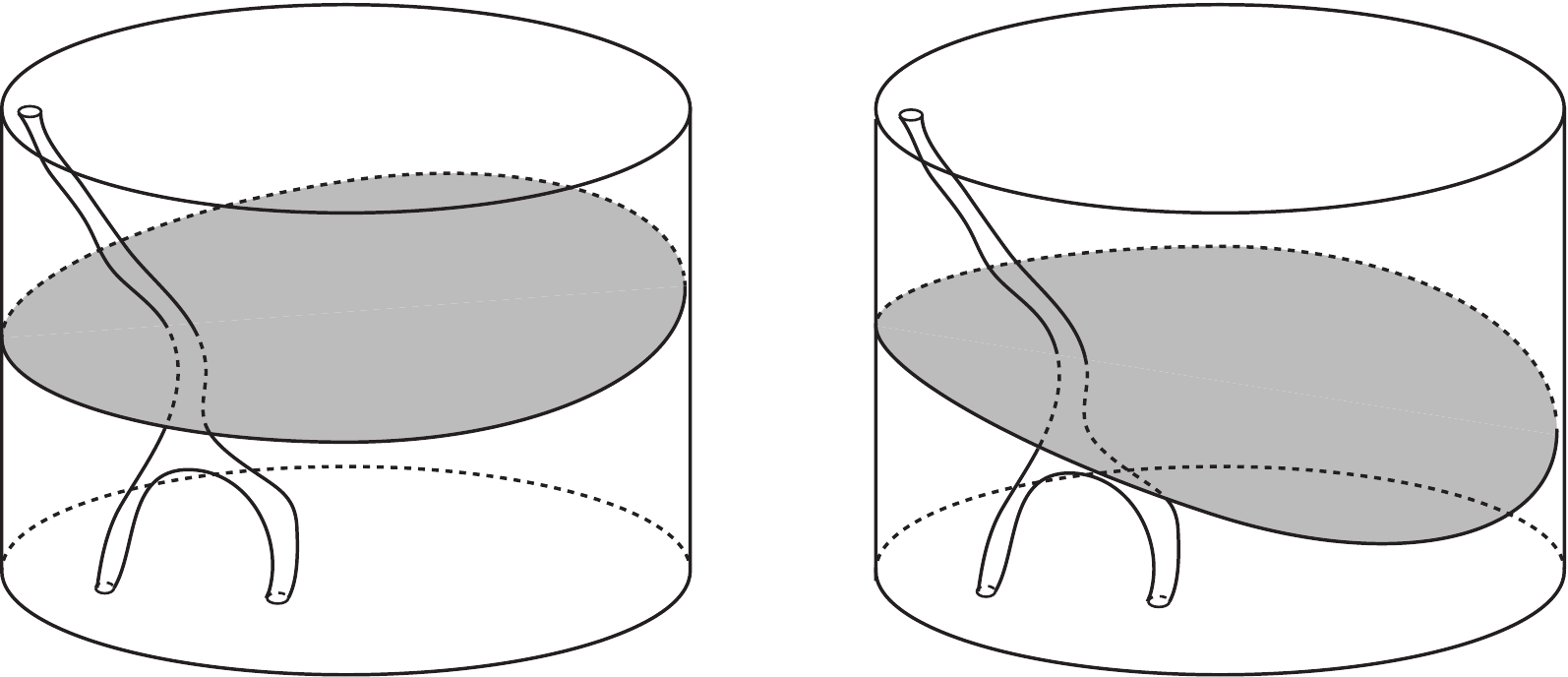}
\put(18,39){$C_n[\R^k]$}
\put(1.5,28.5){$Z$}
\put(26,22){$C_n[\gamma(M)]$}
\put(73,39){$C_n[\R^k]$}
\put(57.5,28.5){$Z$}
\put(82,16){$C_n[i(M)]$}
\end{overpic}
\caption{Here, both $C_n[\gamma(M)]$ and $C_n[i(M)]$ are boundary disjoint with $Z$ in $\bdry C_n[\R^k]$. We can move from $C_n[\gamma(M)]$ to $C_n[i(M)]$ while preserving the transversality of the intersection with $Z$.}
\label{fig:main-idea}
\end{figure}

Given these assumptions, our \thm{transversality-compact-config}  (roughly) says there there is an open dense set of (perturbed) smooth embeddings $\gamma'\co M\hookrightarrow \R^k$, with $C_n[\gamma'(M)]$ transverse to $Z$, and for which $\bdry Z$ and $\bdry C_n[\gamma'(M)]$ are disjoint in $\bdry C_n[\R^k]$.  

In Section~\ref{sect:Haefliger}, we then restrict our attention to smooth embeddings of spheres $S^l$ in $\R^k$, and use Haefliger's Theorem \cite{Haefliger:1961wr} to deduce the existence of a differentiable isotopy between our embeddings $\gamma(S^l)$ and $i(S^l)$. We combine all of these steps together in Section~\ref{sect:main-method} to get our main tool, \thm{main-theorem}, which can be applied to many settings, including the square-peg problem. In \thm{main-theorem}, we assume that 
\begin{compactitem}
\item $\gamma\co S^l\rightarrow \R^k$ is a smooth embedding of $S^l$ in $\R^k$, with a corresponding embedding of compactified configuration spaces $C_n[\gamma]\co C_n[S^l]\rightarrow C_n[\R^k]$;
\item $Z$ is a closed topological space contained in $\cnr$ such that $Z \cap \cnro$ is a submanifold of $\cnro$, and $\bdry Z\subset \bdry\cnr$;
\item  $C_n[\gamma(S^l)]$ and $Z$ are boundary-disjoint;
\item there is a standard smooth embedding  $i\co S^l\hookrightarrow \R^k$, such that $C_n[i(S^l)]$ is transverse to $Z$  in $\cnr$.
\end{compactitem}
We then deduce that  there is, for all $m$, a $C^m$-dense set of smooth embeddings $\gamma'\co S^l\hookrightarrow \R^k$, such that the corresponding embeddings on configuration spaces are $C^0$-close to $C_n[\gamma]$, and that $C_n[\gamma'(S^l)]$ is transverse to $Z$, and moreover $C_n[i(S^l)] \cap Z$ and $C_n[\gamma'(S^l)] \cap Z$ represent the same homology class in $Z$.

We next apply our method, \thm{main-theorem}, to the square-peg problem. In Section~\ref{sect:cfg}, we start by reviewing the structure of $C_4[\gamma(S^1)]$, the compactified configuration space of 4 points on a smooth embedding of a circle in $\R^k$.  We choose $\R^k$, rather than $\R^2$, as we will prove a more general version of Toeplitz's conjecture. Our main result is about the existence of a square-like quadrilateral inscribed in a smooth embedding of a circle in $\R^k$. In Section~\ref{sect:slq}, we define the space of {\em square-like quadrilaterals} in $\R^k$ denoted $\slq$. These are quadrilaterals $abcd$ with equal sides ($|ab|=|bc|=|cd|=|da|$) and equal diagonals ($|ac|=|bd|$). 
We then prove that 
\begin{compactenum}
\item $\slq$ is a submanifold of $C_4(\R^k)$ and $\bdry\slq\subseteq\bdry C_4[\R^k]$, 
\item the boundaries of $\slq$ and $C_4[\gamma(S^1)]$ are disjoint in $\bdry C_4[\R^k]$.
\end{compactenum}  

We next consider the number of intersections of $\slq$ with our standard smooth embedding $i\co S^1\hookrightarrow \R^k$ of a circle in $\R^k$. We choose this embedding to be the planar ellipse $\nicefrac{x^2}{a^2} + \nicefrac{y^2}{b^2} = 1$ with $a > b$. However, when we count the number of intersections of $\slq$ with $C_4[i(S^1)]$, we get an even number of points. Thus, we need to consider intersections counted up to cyclic relabeling. In Section~\ref{sect:cyclic} we carefully construct the quotient spaces needed to make this argument. In Section~\ref{sect:basecase} we show that our ellipses have a  a single inscribed square up to cyclic relabeling, and moreover our quotient spaces intersect transversally at this intersection. Finally, our arguments culminate in \thm{squarepeg}, where we conclude that there is, for all $m$, a $C^m$-dense\footnote{The density is with respect to the Whitney $C^\infty$-topology, described in detail in Section~\ref{sect:transversality}.} set of smooth embeddings of a circle in $\R^k$ each of which has an odd number of square-like quadrilaterals.

It is important to realize that while our results provide a unified and attractive view of this family of theorems about special inscribed configurations, they do not directly address the remaining open territory in Toeplitz's question. In \cite{FTCWC}, we give an extension of our results to prove that there exists at least one square-like quadrilateral inscribed in any embedding of $S^1$ in $\R^k$ which is of finite total curvature without cusps.  (We note that when $k=2$, this class of curves is less general than the family of curves for which W.~Stromquist \cite{MR1045781} and B.~Matschke~\cite{2010arXiv1001.0186M,MR3184501}  proved the square-peg theorem.)

Finally we note that in \cite{Simplices}, we provide another example of our main technique (\thm{main-theorem}), where we show that there is a $k(k-1)/2$ dimensional family of inscribed $(k+1)$-simplices of any constructible edgelength ratio in certain generic smooth embeddings of $S^{k-1}$ in $\R^k$.

      
\section{Configuration Spaces}
\label{sect:config}

The compactified configuration space of $n$ points in $\R^k$ is the natural setting for the square-peg and other inscribed polygon problems. In this section we give a brief overview of the  theory of compactified configuration spaces. There are many versions of this classical material (see for instance~\cite{MR1259368,MR1258919}). We follow Sinha~\cite{newkey119}, as this gives a geometric viewpoint appropriate to our setting. A reader familiar with configuration spaces may skip much of this section. However we recommend paying attention to the notation we have used for the spaces, points in the spaces and the strata.  \defn{pijsijk}, \defn{config}, and \rmark{notation} are particularly useful. 

\begin{definition}[\cite{newkey119}]\label{def:openconfig}
Given an $m$-dimensional smooth manifold $M$, let  $M^{\times n}$ denote the $n$ folk product of copies of $M$, and define $\cnmo$ to be the subspace of points  ${\mathbf{p}}=(p_1,\dots, p_n)\in M^{\times n}$ such that $p_j\neq p_k$ if $j\neq k$. Let $\iota$ denote the inclusion map of $\cnmo$ in~$M^{\times n}$. 
\end{definition}

The space $\cnmo$ is an open submanifold of $M^{\times n}$. Our goal is to compactify $\cnmo$ to a closed manifold-with-boundary and corners, which we will denote $\cnm$, without changing its homotopy type. The resulting manifold will be homeomorphic to $M^{\times n}$ with an open neighborhood of the fat diagonal
removed. Recall that the fat diagonal is the subset of $M^{\times n}$ of $n$-tuples for which (at least)
two entries are equal, that is, where
some collection of points comes together at a single point. The construction of $\cnm$ preserves information about the directions and relative rates of approach of each group of collapsing points.

\begin{definition}[\cite{newkey119},\cite{newkey118}]  \label{def:pijsijk}
Given an ordered pair of distinct elements from $\{1,\dots,n\}$, let the map $\pij\co C_n(\R^k)\rightarrow S^{k-1}$ send~$\bfp=(\bfp_1,\dots\bfp_n)$ to $\displaystyle\frac{\bfp_i-\bfp_j}{\norm{\bfp_i-\bfp_j}}$, the unit vector in the direction of $\bfp_i-\bfp_j$.  Let $[0,\infty]$ be the one-point compactification of $[0,\infty)$.  Given an ordered triple $(i,j,l)$ of distinct elements in $\{1,\dots,n\}$, let $r_{ijl} \co C_n(\R^k)\rightarrow [0,\infty]$ be the map which sends $\bfp$ to $\displaystyle\frac{\norm{{\bfp_i}-{\bfp_j}} }{\norm{ \bfp_i-{\bfp_l}} }$, the ratio of distances between $\bfp_i$ and $\bfp_j$, and $\bfp_i$ and $\bfp_l$. Define $s_{ijl}\co C_n(\R^k)\rightarrow [0,1]$ as the composition $(\frac{2}{\pi}\arctan)\circ ( r_{ijl})$.
\end{definition}

 We then compactify $C_n(\R^k)$ as follows:

\begin{definition} [\cite{newkey119}] \label{def:config} 
\begin{compactenum}\item Let $A_n[\R^k]$ be the product $(\R^k)^{n}\times (S^{k-1})^{n(n-1)} \times [0,1]^{n(n-1)(n-2)}$.  Define $C_n[\R^k]$ to be the closure of the image of $C_n(\R^k)$ under the map
$$\alpha_n= \iota \times (\pij) \times (s_{ijl}) \co C_n(\R^k)\rightarrow A_n[\R^k].$$
\item  We assume that all manifolds $M$ are smoothly embedded in $\R^k$, which allows us to define the restrictions of the maps $\pij$ and $s_{ijl}$.  Then $\cnmo$ is smoothly embedded in $C_n(\R^k)$ and we define $\cnm$ to be the closure of $\cnmo$ in $M^{n}\times (S^{k-1})^{n(n-1)} \times [0,1]^{n(n-1)(n-2)}$. 
We denote the {\em boundary} of $\cnm$ by $\bdry\cnm=\cnm\setminus\cnmo$.
\end{compactenum}
\end{definition}

We now summarize some of the important features of this construction, including the fact that $\cnm$ does not depend on the choice of embedding of $M$ in $\R^k$.

\begin{theorem}[\cite{newkey119},~\cite{newkey118}] \label{thm:config}
\begin{compactenum}
\item[\rm{1.}] $\cnm$ is a manifold-with-boundary and corners with interior $\cnmo$ having the same homotopy type as~$\cnm$. The topological type of $\cnm$ is independent of the embedding of $M$ in $\R^k$, and $\cnm$ is compact when $M$ is.
\item[{\rm 2.}] The inclusion of $\cnmo$ in $M^{\times n}$ extends to a surjective map  from $\cnm$ to $M^{\times n}$ which is a homeomorphism over points in $\cnmo$. 
\end{compactenum}
\end{theorem}

\begin{remark}\label{rmark:notation}
When discussing points in $\cnr$ or $\cnm$, it is easy to become confused. We pause to clarify notation.
\begin{itemize}
\item A point in $\R^k$ is denoted by $\bfx=(x_1, \dots, x_k)$, where each $x_i\in\R$.
\item Points in $(\R^k)^n$ are also denoted by $\bfx$, where $\bfx = (\bfx_1, \dots, \bfx_n)$ and each $\bfx_i\in\R^k$. (It will be clear from context which is meant.)
\item A point in $\cnr$ or $\cnm$, is denoted $\x$. 
\item At times, we will need to distinguish between the various entries of $\x\in\cnr$ or $\cnm$. 
In general,  
$$\x=(\bfx,(\pij)(\bfx), (s_{ijl})(\bfx)) =  (\bfx, \alpha(\bfx)),$$ where $\bfx=(\bfx_1,\dots,\bfx_n)\in (\R^k)^n$, and $ \alpha(\bfx) = ((\pij)(\bfx), (s_{ijl})(\bfx))$ gives the corresponding set of values in $(S^{k-1})^{n(n-1)}$ and $[0,1]^{n(n-1)(n-2)}$.
 \end{itemize}
\end{remark}

The space $\cnm$ may be viewed as a polytope with a combinatorial structure based on the different ways groups of points in $M$ can come together. This structure defines a stratification of $\cnm$ into a collection of closed faces of various dimensions whose intersections are members of the collection. We will use a bit of the structure of this collection, referred to as a \emph{stratification} of $\cnm$.

\begin{definition}[\cite{newkey118}] A \emph{\paren} $\mathcal{P}$ of a set $T$ is an unordered collection $\{A_1,\dots, A_l\}$ of subsets of $T$ such that $\#A_s\geq 2$, and for $s\neq t$ either $A_s\cap A_t=\emptyset$, or $A_s\subset A_t$, or $A_t\subset A_s$. A \paren is denoted by a nested listing of the $A_s$ using parentheses. Let ${\bf Pa}(T)$ denote the set of \parens of $T$, and define an ordering on it by $\mathcal{P}\leq\mathcal{P}'$ if $\mathcal{P}\subseteq\mathcal{P}'$.
\end{definition}
For example, for $T=\{1,2,3,4\}$, $(12)(34)$ represents a \paren whose subsets are $\{1,2\}$ and $\{3,4\}$ while $((12)34)$ represents a \paren whose subsets are $\{1,2\}$ and $\{1,2,3,4\}$. 

We identify each parenthesization $\mathcal{P} = \{A_1, \dots, A_l\}$ of $\{1, \ldots, n\}$ with a closed subset $S_\mathcal{P}$ of $\bdy \cnm$ in our stratification of $\cnm$. The idea is that all the points in each $A_s$ collapse together, but if $A_s \subset A_t$, then the points in $A_s$ collapse ``faster'' than the points in $A_t$. Formally, this becomes the following condition: Let $\p = ((\bfp_1\dots,\bfp_n),(\pij)(\bfp), (s_{ijl})(\bfp))$ be a point in $A_n[M]$. Then $\p \in S_\mathcal{P}$ if 
\begin{itemize}
\item $\bfp_i = \bfp_j$ if and only if $i, j \in A_s$ for some $s$.
\item $s_{ijl} = 0$ (and hence $s_{ilj} = 1$) if and only if $i, j \in A_s$ and $l \notin A_s$ (see  Proposition 3.3, Definition 2.10 \cite{newkey119}).
\end{itemize}
Sinha~\cite{newkey119} proves that a stratum $S_\mathcal{P}$ described by nested subsets $\{A_1,\dots, A_l\}$ has codimension~$l$ in $\cnm$. In the previous example $(12)$ has codimension~1, while $((12)34)$ and $(12)(34)$ have codimension~2. 

Any pair $\bfp$, $\bfq$ of disjoint points in $\R^k$ has a direction $(\bfp-\bfq)/\norm{\bfp-\bfq}$ associated to it, while every triple of disjoint points $\bfp$, $\bfq$, $\bfr$ has a corresponding distance ratio $\norm{\bfp-\bfq}/\norm{\bfp-\bfr}$. One way to think of the coordinates of $\cnm$ is that they extend the definition of these directions and ratios to the boundary.

\begin{theorem} [\cite{newkey119}, \cite{newkey118}] \label{thm:configgeom} Given a manifold $M \subset \R^k$, then in any configuration of points $\p \in \cnm$ the following holds.
\begin{compactenum}
\item  Each pair of points $\bfp_i$, $\bfp_j$ has associated to it a well-defined unit vector in $\R^k$ giving the direction from $\bfp_i$ to $\bfp_j$. If the pair of points project to the same point $\bfp$ of $M$, this vector lies in $T_\bfp M$. 
\item Each triple of points $\bfp_i$, $\bfp_j$, $\bfp_k$ has associated to it a well-defined scalar in $[0,\infty]$ corresponding to the ratio of the distances $\norm{\bfp_i - \bfp_j}$ and $\norm{\bfp_i - \bfp_k}$. If any pair of $\{\bfp_i, \bfp_j,\bfp_k\}$ projects to the same point in $M$ (or all three do), this ratio is a limiting ratio of distances.
\item The functions $\pi_{ij}$ and $s_{ijl}$ are continuous on all of $\cnm$ and smooth on each face of $\bdy\cnm$. 
\end{compactenum}
\end{theorem}

We notice that the definition of the $S_{\mathcal{P}}$ does not depend on the $\pi_{ij}$. In fact, for connected manifolds of dimension at least $2$, the combinatorial structure of the strata of $\cnm$ depends only on the number of points. Regardless of dimension, this construction and division of $\bdy \cnm$ into strata is functorial in the following sense.

\begin{theorem}[\cite{newkey119}]\label{thm:functor}
Suppose $M$ and $N$ are embedded submanifolds of $\R^k$ and $f\co  M\hookrightarrow N$ is an embedding. This induces an embedding of manifolds-with-corners called the evaluation map $C_n[f] \co C_n[M]\hookrightarrow C_n[N]$ that respects the stratifications.  This map is defined by choosing the ambient embedding of $M$ in $\R^k$ to be the composition of $f$ with
the ambient embedding of $N$.
\end{theorem}

For an embedding $f \co M \hookrightarrow N$, the image of the induced embedding $C_n[f]\co C_n[M] \hookrightarrow C_n[N]$ will be denoted by $C_n[f(M)]$.

\begin{corollary}\label{cor:smooth}
Let $f\co\R^k\rightarrow \R^k$ be a smooth diffeomorphism. Then the induced map of configuration spaces $\cnf\co
\cnr\rightarrow\cnr$ is also a smooth diffeomorphism (on each face of $\cnr$).
\end{corollary}

\begin{proof} This is an immediate corollary of 
the previous theorem.
\end{proof}

\section{Configuration Spaces and Transversality}
\label{sect:transversality}

In this section, we prove a transversality ``lifting property'' for compactified configuration spaces: The submanifold of configurations of points on a smoothly embedded submanifold $M$ of $\R^k$ may be made transverse to any submanifold $Z$ of the configuration space of points in $\R^k$ by an arbitrarily small variation of $M$, as long as the two submanifolds of configuration space are boundary-disjoint. This is a useful technique and parts of it have been proved before. For instance, R.~Budney et al.~\cite{newkey118} prove a special case of this result. 
We will show that a general form of this result may be obtained easily from the Multijet Transversality Theorem~(\cite{Golubitsky:1974iu},  Theorem II.4.13).

We begin by recalling some details about the construction of jet space and the Whitney $C^\infty$-topology on mappings. Then we will state the multijet transversality theorem and show that our desired result on configuration space transversality follows. 

\begin{definition}
\label{def:jetspace}
Let $M$ and $N$ be smooth manifolds, and $f$ be a smooth function $f \co M \rightarrow N$. The {\em space of $0$-jets} $J^0(M,N) = M \cross N$. The {\em $0$-jet of $f$} is the function $j^0 f \co M \rightarrow J^0(M,N)$ given by $j^0 f(\bfp) = (\bfp,f(\bfp))$.
\end{definition}

It is a standard fact that jet space $J^0(M,N)$ is a smooth manifold. Further, $0$-jet spaces may be extended to $k$-jet spaces by an inductive procedure involving taking successive derivatives. We won't need higher jet spaces here; we refer the interested reader to \cite{Golubitsky:1974iu} for details. 

We can extend the definition of jet space to a space of $n$-fold \emph{multijets} as follows.

\begin{definition}
\label{def:multijets} Let the source map $\sigma \co J^0(M,N)^{\times n} \rightarrow M^{\times n}$ be given by  
$$ \sigma((\bfp_1,\bfq_1),\dots ,(\bfp_n,\bfq_n)) = (\bfp_1,\dots , \bfp_n).$$
 Define the space of  {\em $n$-fold $0$-multijets} $J_n^0(M,N)=\sigma^{-1}(C_n(M))$.
 Given a smooth function $f \co M \rightarrow N$, there is a natural smooth map $j^0_n f \co \cnmo \rightarrow J_n^0(M,N)$ given by
\begin{equation*}
j^0_n f(\bfp) = (j^0 f(\bfp_1), \dots, j^0 f(\bfp_n))= ((\bfp_1,f(\bfp_1)),\dots , (\bfp_n,f(\bfp_n)) ).
\end{equation*}
\end{definition}


The space $C^\infty(M,N)$ has the Whitney $C^\infty$-topology. Recall from \cite{Hirsch}, for $r$ finite, the $C^r$-topology on $C^\infty(M,N)$ has as subbasis sets of the form
$$\mathcal{N}^r(f; (U,\phi), (V,\psi), \delta).$$
This denotes the subset of functions $g\co M\to N$ that are smooth, and for coordinate charts $\phi\co (U' \subset M) \to (U\subset \R^m)$ and $\psi\co 
(V' \subset N) \to (V\subset \R^k)$ and $K \subset U$ compact with $g(\phi^{-1}(K)) \subset V'$, then, for all $s \leq r$, and all $\bfx \in K$,
$$\| D^s (\psi g \phi^{-1})(\bfx) - D^s(\psi f \phi^{-1})(\bfx)\| < \delta. $$
Here, $D^s F$ for a function $F \co (U\subset \R^m) \to (V \subset \R^k)$ is the $k$-tuple of the $s$th  homogeneous parts of the Taylor series representations of
the projections of $F$. The topology generated by this subbasis is the Whitney $C^r$-topology on $C^r(M,N)$. The subspace $C^\infty(M,N)$ has the Whitney
$C^\infty$-topology by taking the union of all subbases for all $r\geq 0$.

For $M$ compact, in this topology, it is a standard theorem that the subset $\mbox{\rm Emb}^\infty(M,N) \subset C^\infty(M,N)$ of smooth embeddings of $M$ into $N$ is an open set (see Theorem 2.1.4 of \cite{Hirsch}). Another important topological result holds for 0-multijets:

\begin{theorem}[{\bf 0-Multijet Transversality Theorem}, \cite{Golubitsky:1974iu} Theorem II.4.13]
Let $M$ and $N$ be smooth manifolds and let $Z$ be a submanifold of $J_n^0(M,N)$. Let 
\begin{equation*}
T_Z = \left\{ f \in C^\infty(M,N) \mid j_n^0 f \transverse Z \right\}.
\end{equation*}
Then $T_Z$ is $C^m$-dense in $C^\infty(M,N)$ for any $m$. Moreover, if $Z$ is compact, then $T_Z$ is $C^\infty$-open in $C^\infty(M,N)$. 
\label{thm:multijet}
\end{theorem}

Before proceeding any further, we recall the definition of transversality following \cite{Guillemin:2010ti}.  Assume that $f\co  X\rightarrow Y$  is a map between manifolds, and $Z$ is a submanifold of $Y$.  Let $x\in f^{-1}(Z)$ and $y=f(x)$. Then $f$ is said to be {\em transversal} to $Z$, denoted $f\pitchfork Z$, provided that $\text{Image}(Df_x) + T_y(Z) = T_y(Y)$ holds for each $x\in f^{-1}(Z)$. Moreover, following \cite{Guillemin:2010ti}, we know that if $f\pitchfork Z$, then $f^{-1}(Z)$ is a submanifold of $X$.

We note that  \thm{multijet} is actually a bit stronger than the version we have stated. In fact, it shows $T_Z$ is a \emph{residual} set,  that is, a countable intersection of open dense subsets of $C^\infty(M,N)$. 

The definition of compactified configuration spaces allows us to view $C_n[N] \subset (\R^k)^{n} \times (S^{k-1})^{n(n-1)} \times [0,1]^{n(n-1)(n-2)}$ as a metric space with the sup norm.  If we define the mapping $\mbox{\rm pr}_i$ to be the projection onto the $i$th space of the product, then this naturally leads to a metric on the set of continuous functions $C^0(C_n[M], C_n[N])$.
\begin{definition} With the above assumptions, the metric on the set  $C^0(C_n[M], C_n[N])$ is given by 
$$\|F - G\|_0 = \sup_{\p \in C_n[M]} \{\| \mbox{\rm pr}_i(F(\p)) - \mbox{\rm pr}_i(G(\p))\| \mid \mbox{\rm for all\ }i\}.$$
Thus, given a maps $f, g\co M\rightarrow N$ and $\epsilon >0$,  we say that the corresponding maps on configuration spaces $C_n[f],C_n[g]\co \cnm\rightarrow \cnn$ are {\em $\epsilon$-close} provided $\|C_n[f]-C_n[g]\|_0<\epsilon$.
\end{definition}

In order to prove a useful result for configuration spaces with special submanifolds, we first prove the following lemma: 

\begin{lemma}\label{lem:continuous}
For $M$ compact, the mapping $C_n[\ ] \co C^\infty(M,N) \to C^0(C_n[M], C_n[N])$ is continuous.
\end{lemma}

\begin{proof}
Because $M$ is compact, so is $C_n[M]$. As discussed above, we take $C_n[N]$ with the  metric topology as a subspace of $A_n[\R^k]$. The metric topology on $C^0(C_n[M], C_n[N])$ is
the topology of compact convergence which coincides with the compact-open topology (\cite{Munkres}). This implies that $C_n[\ ] \co C^\infty(M,N) \to C^0(C_n[M], C_n[N])$ is continuous if and only if the adjoint $\widehat{C_n[\ \,]} \co C_n[M] \times C^\infty(M,N) \to C_n[N]$ is continuous. Since $C_n[N]$ as a subspace of
the product $(\R^k)^n \times (S^{k-1})^{n(n-1)} \times [0,1]^{n(n-1)(n-2)}$, then maps into $C_n[N]$ are continuous if and only if the compositions with projections onto a factor are continuous. Consider the composition:
$$C_n[M] \times C^\infty(M,N)  \mapright{\widehat{C_n[\ \,]}}{}  C_n[N] \mapright{\mbox{pr}_j}{} F'_j.$$
On the factors $F'_j = N$, the composition is evaluation on the corresponding factor of $C_n[M]$ and hence is continuous. When the factor is a sphere, on points $\bfp_i \neq
\bfp_j$ with $f(\bfp_i) \neq f(\bfp_j)$, the map is simply the composition of a function $f$ with $\pi_{ij}$ and hence is continuous. When points come together, either in $M$ or in $f(M)$, the image is in one of the strata of $C_n[N]$. In this case, we can view $\bfp_i = \bfp_j + t\vic{u}$ for $\vic{u}$ a unit vector. Then
$$\dfrac{f(\bfp_j + t\vic{u}) - f(\bfp_j)}{\|f(\bfp_j + t\vic{u}) - f(\bfp_j)\|} \quad \mathop{\longrightarrow}\limits_{t\to 0} \quad \dfrac1{\norm{\det Df(\bfp_j)} } Df(\bfp_j)(\vic{u}).$$
On $C^\infty(M,N)$ this composition is continuous.

Finally, if $F_j' = [0,1]$, then the composition is similarly analyzed using $s_{ijl}$ instead of $\pi_{ij}$ to establish continuity.
\end{proof}

%

\begin{corollary}
\label{cor:epsilonclose}
Let $M$ and $N$ be smooth manifolds, and assume $M$ is compact. Given $\epsilon > 0$ and $f\co M \hookrightarrow N \subset \R^k$ an embedding, there is an open set $U\subset C^\infty(M,N)$ containing $f$ such that for  $g \in \mbox{\rm Emb}^\infty(M,N) \cap U$,  the maps $C_n[f]$ and $C_n[g]$ are $\epsilon$-close: $\| C_n[f] - C_n[g] \|_0 < \epsilon.$
\end{corollary}

\begin{proof}
Let $U$ be the preimage of the open set of functions $C_n[M] \hookrightarrow C_n[N]$ within $\epsilon$ of $C_n[f]$ under $C_n[\ ]$. This is an open subset of $C^\infty(M,N)$ containing $f$ by the previous lemma. Since $\mbox{\rm Emb}^\infty(M,N)$ is open in $C^\infty(M,N)$, it meets $U$ in an open set. Choose $g$ in this open set.
\end{proof}

Let us pause to appreciate what we have proven here. Given an embedding $f\co M \hookrightarrow N$, where $M$ is compact, we can find a smooth embedding $g\co M \hookrightarrow N$ in a $C^\infty$-neighborhood of $f$, such that the corresponding maps between configuration spaces $C_n[f]$ and $C_n[g]$ are as $C^0$-close as we like. Note that more is probably true, for example that $C_n[f]$ and $C_n[g]$ are $C^m$-close for $m\ge 1$. However, we do not need such a result and have not proved it here.

Next, we see that \cor{epsilonclose} leads to two very useful results.

\begin{theorem}[{\bf Transversality Theorem for Configuration Spaces}]
\label{thm:transversality-config}
Let $M$ and $N$ be smooth manifolds with $M$ compact,  and $i \co M \hookrightarrow N \subset \R^k$  a smooth embedding with
corresponding embedding of configuration spaces $C_n(i)\co \cnmo\hookrightarrow\cnno$. Assume $Z$ is a submanifold of $\cnno$. Given an $\epsilon > 0$, there is a $C^\infty$-open neighborhood $U$ of $i$, in which there is, for all $m$, a $C^m$-dense set of smooth embeddings  $i'\co M\hookrightarrow N$,  with $\| C_n[i'] - C_n[i] \|_0 < \epsilon$, and $C_n(i')\pitchfork Z$. 
\end{theorem}

\begin{proof} 
We embed $Z$ in $\cnmno$ as $W=sh(\cnmo\times Z)$, where the shuffle map is defined by $sh((p_1,\dots,p_n),(v_1,\dots,v_n)) = ((p_1,v_1),(p_1,v_1), \dots, (p_n,v_n))$. This map is a diffeomorphism. Using \thm{multijet}, we know that the set $T_W=\{f\in C^\infty(M,N) \, \vert \, j_n^0f\transverse W\}$ is $C^m$-dense in $C^\infty(M,N)$ for all $m$.

The preimage of the $\epsilon$-ball around $C_n(i)$ is an open set $U$ in $C^\infty(M,N)$ containing $i$. Since $\mbox{\rm Emb}^\infty(M,N)$ is open, so is $U \cap \mbox{\rm Emb}^\infty(M,N)$, and this open set meets the set $T_W$. Choose $i'$ in the intersection $T_W \cap U \cap \mbox{\rm Emb}^\infty(M,N)$. 
By Theorem~\ref{thm:multijet}, $j_n^0i' \pitchfork W$, and the set of such $i'$ remains $C^m$-dense in $C^\infty(M,N)$ for all $m$.

We want to show that this implies $C_n(i') \pitchfork Z$. By definition, $j_n^0i'$ transverse to $W$ means that for all $\bfp$ with $j_n^0i'(\bfp)\in j_n^0(i')(M)\cap W$, 
$$ T_{j_n^0i'(\bfp)}\cnmno \cong T_{j_n^0i'(\bfp)}W \oplus Dj_n^0i'(T_\bfp\cnmo).
$$
Since $W\cong C_n(M)\times Z$, then $T_{j_n^0i'(\bfp)} W \cong T_\bfp\cnmo \oplus T_{C_n(i')(\bfp)}Z.$
The key step here is to use the shuffle map to rewrite the decomposition.
Since,
$$j_n^0i'(\bfp)_ = ((\bfp_1,i'\bfp_1), \dots ,(\bfp_n,i'\bfp_n)) \xrightarrow{sh^{-1}} ((\bfp_1,\dots,\bfp_n),(i'\bfp_1,\dots,i'\bfp_n))\in C_n(M)\times C_n(N).
$$
then
$$Dj_n^0i'(T_\bfp\cnmo) \cong T_\bfp\cnmo \oplus Di'(T_\bfp \cnmo).
$$
The shuffle maps also allows us to identify $$ T_{j_n^0i'(\bfp)}\cnmno \cong T_\bfp\cnmo \oplus T_{C_n(i')(\bfp)}\cnno.
$$
Putting all the maps together, we deduce 
$$  T_{C_n(i')(\bfp)}\cnno \cong  T_{C_n(i')(\bfp)}Z \oplus  Di'(T_\bfp \cnmo).$$
In other words, $C_n(i')\pitchfork Z$.
\end{proof}

\begin{theorem}[{\bf Transversality Theorem for Compactified Configuration Spaces}]\label{thm:transversality-compact-config}
Let $M$ and $N$ be smooth manifolds, with $M$ compact, and $i \co M \hookrightarrow N \subset \R^k$ a smooth embedding with a 
corresponding embedding of compactified configuration spaces $C_n[i]\co \cnm\hookrightarrow\cnn$. Assume $Z$ is a closed topological space contained in $\cnn$, such that $Z\cap \cnno$ is a submanifold of $\cnno$, and  $\bdry Z\subset \bdry \cnn$. Also assume that $\bdry Z$ is disjoint from $\bdry C_n[i(M)]$. 

Then for any $\epsilon > 0$,  there is a $C^\infty$-open neighborhood of $i$, in which there is, for all $m$, a $C^m$-dense set of smooth embeddings  $i'\co M\hookrightarrow N$,  with $\| C_n[i'] - C_n[i] \|_0 < \epsilon$, and $C_n[i']\pitchfork Z$, and for which $\bdry Z$ and $\bdry C_n[i'(M)]$ are disjoint in $\bdry\cnn$.
\end{theorem}

\begin{proof} Recall that $M$ and $N$ are embedded in $\R^k$ for $k$ large. Since $M$ is compact, the closed set $C_n[i(M)]\cap Z$ is also compact. By assumption, $C_n[i(M)]\cap Z$ is also disjoint from the closed set $\bdry \cnn$, thus it is separated from it by some $\epsilon >0$. Take the intersection of $Z$ with the complement of an $\epsilon/2$ neighborhood of $\bdry \cnn$, and denote this $\hat{Z}$.  The set $\hat{Z}$ is an open manifold contained in $\cnno$, which remains a bounded distance from $\bdry \cnn$.

We apply \thm{transversality-config} to this setting, so there is an embedding $i'\co M\hookrightarrow N$ 
with  $C_n[i']\co \cnmo \hookrightarrow \cnno$ transverse to $\hat{Z}$, which by \cor{epsilonclose}, we can choose so that  
$C_n[i'(M)]$ is $\epsilon/2$-close to $C_n[i(M)]$. We can thus choose the perturbation $i'$ so that  $C_n[i'(M)]\cap \hat{Z}$ is at least $(\nicefrac{3}{4})\epsilon$ away from $\bdry \cnn$. Thus $C_n[i']\pitchfork \hat{Z}$ implies $C_n[i'(M)]\pitchfork \hat{Z}$ which in turn implies $C_n[i'(M)]\pitchfork Z$.
\end{proof}

There are many possible applications of \thm{transversality-compact-config}. The first one we give is to the square-peg problem in this paper, another is to $(k+1)$-simplices inscribed in smoothly embedded $(k-1)$-spheres found in \cite{Simplices}. Thus, we let $M=S^l$ be an $l$-sphere embedded in $N=\R^k$, and $Z$ is a submanifold of $\cnr$ satisfying special conditions. For our work $C_n[i(S^l)]\cap Z$ will usually represent certain inscribed configurations of points in $S^l$. In this setting \thm{transversality-compact-config} becomes the following.

\begin{corollary}
\label{cor:transversality}
Suppose there is a smooth embedding $i\co S^l\hookrightarrow\R^k$ of an $l$-sphere in $\R^k$, with a corresponding embedding of compactified configuration spaces $C_n[i]\co \cnsl \hookrightarrow \cnr$.  Assume that $Z$ is a closed topological space contained in $\cnr$ such that $Z \cap \cnro$ is a submanifold of $\cnro$, and $\bdry Z\subset \bdry\cnr$. Also assume that  $\bdry Z$ is disjoint from $\bdry C_n[i(S^l)]$. 

 Then for any $\epsilon > 0$,  there is a $C^\infty$-open neighborhood of $i$, in which there is, for all $m$, a $C^m$-dense set of smooth embeddings  $i'\co S^l \hookrightarrow \R^k$,  with $\| C_n[i'] - C_n[i] \|_0 < \epsilon$, and $C_n[i']\pitchfork Z$, and for which $\bdry Z$ and $\bdy C_n[i'(S^l)]$ are disjoint in $\bdy \cnr$. 
\end{corollary}

\subsection{Deformations}\label{sect:Haefliger}
In this subsection, we restrict our attention to $S^l$, an $l$-sphere embedded in $\R^k$. We want to be able to deform standard spheres into spheres of interest, and then consider what happens on the level of configuration spaces.  We know such a deformation of spheres exists due to a result of A.~Haefliger, which we have stated in a form useful to us (actually, his result is stronger). 

\begin{theorem}
{\rm\cite{Haefliger:1961wr}} Any two differentiable embeddings of $S^l$ in $\R^k$ are homotopic through a differentiable
isotopy in $\R^K \supset \R^k$ when $K > 3(l+1)/2$.
\label{thm:isotopy}
\end{theorem}

Generally such an isotopy must pass through spheres embedded in a higher-dimensional space, as when the
spheres are knotted. The classical case of knots in $\R^3$ requires embedding a knot in $\R^4$ for unknotting. 
Since differentiable knotting is stronger than topological knotting and we prefer to work in
the differentiable category, we need even more extra room to work.\footnote{As before.} In this case, simply
use the usual embeddings $\R^k \hookrightarrow \R^{k+1}  \hookrightarrow \cdots  \hookrightarrow \R^K$
to achieve $K > 3(l+1)/2$. and $\R^k \oplus \vec{0}  \hookrightarrow \R^K$.

There are some useful generalizations of the theorem of Haefliger that may be applied
to obtain isotopies between embeddings. The foundational example, the Whitney-Wu Unknotting Theorem~\cite{MR137124,MR1503303,MR104272}, states
that if $N$ is a compact, connected $n$-manifold with $n\geq 2$ and $m \geq 2n +1$ then any two embeddings
of $N$ into $\R^m$ are isotopic. In the cases of interest in this paper, $N = S^l$ provides a geometrically
satisfying setting for the kinds of special inscribed configurations we study. It is possible to extend our methods
to other manifolds embedded in Euclidean space for which the existence of special inscribed configurations 
may be more difficult and the initial conditions for our techniques harder to find. We leave this to the reader, as the robustness of transversality
arguments cannot be underestimated.


Suppose  $Z$ is a subspace of $\cnr$, typically defined by geometric conditions. We want to understand what happens to the intersection of $Z$ with the configuration spaces of the (isotopic) embedded spheres. It turns out that the homology classes are preserved.

\begin{theorem} Suppose there are two embeddings $\eta, i:S^l\hookrightarrow \R^k$ of an $l$-sphere in $\R^k$. Assume that $Z$ is a closed topological space contained in $\cnr$ such that $Z \cap \cnro$ is a submanifold of $\cnro$, $\bdry Z\subset \bdry\cnr$, and $\bdry Z$ is disjoint from $\bdry C_n[i(S^l)]$.  Also assume that both $C_n[i]$ and $C_n[\eta]$ are transverse to $Z$. Then in $Z$,  the homology class of $C_n[i(S^l)]\cap Z$ and $C_n[\eta(S^l)]\cap Z$  are equal.
\label{thm:homology}
\end{theorem}

\begin{proof}
By Haefliger's Theorem $\eta$ is homotopic to $i$ because they are in the same path component of $\mbox{\rm Emb}(S^l,\R^k)$, as long as $k>3(l+1)/2$. By functorality, this gives a homotopy $H\co \cnslo \times I\rightarrow \cnro$, and $H(-,0)$ and $H(-,1)$ are both transverse to $Z$.

After applying the Transversality Homotopy Extension theorem (see for instance  \cite{Guillemin:2010ti}), we get a map homotopic to $H$:
$$H'\co\cnslo\times I \rightarrow \cnro,$$
with $H'(-,0)=H(-,0)$, $H'(-,1)=H(-,1)$, and $H'$ is transverse to  $Z$.

We conclude by transversality that the mod 2 intersection numbers are equal. In addition, we conclude that in $Z$ the homology class of $C_n[i]\cap Z$ and $C_n[\eta]\cap Z$ are equal.

Haefliger's theorem requires $k> 3(l+1)/2$. As noted before, if this is not the case, use the usual embedding $\R^k \hookrightarrow\R^K$,  to achieve $K>3(l+1)/2$. Then $H':\cnslo\times I \rightarrow C_n(\R^K)$ still agrees with $H$ on $H(-,0)$ and $H(-,1)$ in the version $\R^k\oplus \overrightarrow{0}\hookrightarrow \R^K$. Since transversality does not change the intersection number mod 2, nor the homology class, the conclusions stand.
\end{proof}

\subsection{Application to configuration spaces}
\label{sect:main-method}

In the rest of the paper, we apply the results of this section to the square-peg problem using the following steps. These steps can also be used to solve other problems such as finding inscribed simplices in spheres as found in \cite{Simplices}. We have thus specialized the arguments from the previous section to the case where $M$ is a sphere $S^l$, and show any smooth embedding $\gamma$ of $S^l$ in $\R^k$ has a neighborhood in which there is a dense set of smooth embeddings $\gamma'$ for which $C_n[\gamma']$ is guaranteed to have certain intersections with various ``target'' submanifolds of $\cnr$ defined by geometric conditions. This restates the idea that a dense set of embeddings of $S^l$ always contain certain inscribed configurations of points.

\begin{description}
\item[Step 0] For any smooth embedding $\gamma:S^l\hookrightarrow \R^k$, view $C_n[\gamma(S^l)]$ as a submanifold of $\cnr$. 

\item[Step 1] Show that certain tuples of points in $\R^k$ satisfying a geometric condition, $Z$,  are a submanifold in $\cnro$, and with $\bdry Z\subset \bdry \cnr$.  Prove that $C_n[\gamma(S^l)]$ and $Z$ are boundary-disjoint.

\item[Step 2] For a standard embedding $i\co S^l\hookrightarrow \R^k$, establish the existence of a transverse intersection between $C_n[i(S^l)]$ and $Z$ inside $\cnr$ (in other words, $C_n[i]\pitchfork Z$). Compute the homology class of the intersection $C_n[i(S^l)] \cap Z$ in~$Z$.

\item[Step 3] Use our transversality theorem (Corollary~\ref{cor:transversality}) to find,  for any $\epsilon >0$,  a  $C^\infty$-open neighborhood of $\gamma$ in which there is, for all $m$, a $C^m$-dense set of smooth embeddings $\gamma'\co S^l \hookrightarrow \R^k$, such that $\|C_n[\gamma'] - C_n[\gamma]\|_0< \epsilon$, and $C_n[\gamma'] \transverse Z$, and for which $\bdry Z$ and $\bdy C_n[\gamma'(S^l)]$ are disjoint in $\bdy \cnr$. 
 
\item[Step 4] Use Haefliger's theorem (\thm{isotopy}) 
to find a smooth map $E\co S^l \cross I \rightarrow \R^K$ with $E(-,0)=i$ our standard embedding and $E(-,1) = \gamma'$ (where $K$ may be greater than our original $k$). 
 Following \thm{homology}, conclude that the intersections $C_n[i(S^l)] \cap Z$ and $C_n[\gamma'(S^l)] \cap Z$  represent the same homology class in $Z$.
\end{description}

By putting all these steps together, we prove the following theorem:
\begin{theorem} \label{thm:main-theorem}
Suppose $\gamma\co S^l\rightarrow \R^k$ is a smooth embedding of $S^l$ in $\R^k$, with a corresponding embedding of compactified configuration spaces $C_n[\gamma]\co C_n[S^l]\rightarrow C_n[\R^k]$.
Assume that $Z$ is a closed topological space contained in $\cnr$ such that $Z \cap \cnro$ is a submanifold of $\cnro$, and $\bdry Z\subset \bdry\cnr$. Also assume that $C_n[\gamma(S^l)]$ and $Z$ are boundary-disjoint.  Suppose there is a standard embedding $i\co S^l\hookrightarrow \R^k$, such that $C_n[i]\pitchfork Z$ in $\cnr$. 

Then for all $\epsilon >0$, there is a  $C^\infty$-open neighborhood of $\gamma$, in which there is, for all $m$, a $C^m$-dense set of smooth embeddings $\gamma':S^l\hookrightarrow \R^k$, such that $\| C_n[\gamma'] - C_n[\gamma]\|_0< \epsilon$, and $C_n[\gamma'] \transverse Z$, and moreover, $C_n[i(S^l)] \cap Z$ and $C_n[\gamma'(S^l)] \cap Z$ represent the same homology class in $Z$. 
\end{theorem}

As a simple example of this, if the standard embedding $C_n[i(S^l)]$ has nonzero intersection with $Z$, then we know that $C_n[\gamma'(S^l)]$ has nonzero intersection with $Z$ as well.

Before moving on to the square-peg problem, we note that the steps outlined above work for the more general setting of smooth embeddings of manifolds $M$ in $N\subset\R^k$ where $M$ is compact. There are many kinds of problems that could be solved using this technology.


\section{The square-peg problem}\label{sect:square-peg}

In this section we will apply the method given in Section~\ref{sect:main-method} to prove a version of the square-peg theorem.  Most of our effort will be put into the initial steps where we define the spaces involved. To do this we first give a detailed description of two submanifolds of $\cnr$. For the first, we take an embedding $\gamma\co S^1\hookrightarrow \R^k$, and consider $C_4^0[\gamma(S^1)]$ which is the submanifold of $4$-tuples on a curve $\gamma$ where the points occur in order according to the orientation of the curve. The second is the submanifold $\slq\subset\cfr$, which is the submanifold of configurations of  square-like quadrilaterals. These are  $4$-tuples of points with equal ``sides''  and equal ``diagonals'' (explained in Section~\ref{sect:slq} below). When $k=2$, these are squares, hence our setting generalizes the square-peg problem. A moment's thought shows that there is a square-like quadrilateral inscribed in $\gamma$ when $\slq$ and $C_4^0[\gamma(S^1)]$ intersect. In fact, we will show that when 
this intersection is transverse, the number of intersections is an odd multiple of 4; thus giving an odd number of inscribed squares-like-quadrilaterals up to cyclic relabeling. 

\subsection{The configuration space of points on a curve.}
\label{sect:cfg}
We now use the results from Section~\ref{sect:config} to familiarize ourselves with the configuration space of $n$ points on an embedded circle in $\R^k$. We will always assume that our embeddings are {\em regular}, that is the tangent vector is nowhere zero. (Otherwise it is possible to smoothly describe an embedded curve with corners, by allowing the tangent vector to smoothly change to zero at each corner.)

\begin{definition}
Let $\gamma$  be a $C^\infty$-smooth embedding of $S^1$ in $\R^k$, with $\cng\co\cns \hookrightarrow \cnr$  the corresponding embedding on compactified configuration spaces. We abuse notation by using~$\gamma$ to mean either the embedding or its image in~$\R^k$. We use $C_n[\gamma(S^1)]$ to mean the compactified configuration space of $n$ points on the simple closed curve $\gamma(S^1)\in \R^k$. 
\end{definition}

By \thm{functor} we know that  $C_n[\gamma(S^1)]$ is a submanifold of $\cnr$ and $\bdry C_n[\gamma(S^1)]\subseteq \bdry\cnr$ with the stratifications respected.  The coordinates for $C_n[\gamma(S^1)]$ are similar to those described in \thm{configgeom}, as they are the image of the coordinates under $\gamma\co  S^1\hookrightarrow \R^k$. I.~Voli\'c~\cite{MR2300426} and Budney {\it et al.}~\cite{newkey118} have detailed descriptions of the coordinates for codimension 1 strata. To give an example, observe that the map $C_n[\gamma]$ takes $(\bfp_1,\dots,\bfp_n)\in\cnso$  to $(\gamma(\bfp_1),\dots,\gamma(\bfp_n))\in \cnro$. If we consider the stratum where say $\bfp_1$, $\bfp_2$ and $\bfp_3$ degenerate to a point $\q$ in $\cns$, then $\q$ is a configuration of $n-3+1=n-2$ points plus the $\pij$ and $s_{ijl}$ information for $\bfp_1$, $\bfp_2$ and $\bfp_3$. In $\cnr$ we get a configuration of $n-2$ points on  $\gamma$ plus the directions of approach of the colliding~$\gamma(\bfp_i)$ and the relative distances $s_{123}$, $s_{312}$, and so forth. The $\pi_{ij}$ are unit tangent vectors to $\gamma$. If $\bfp_1$ and $\bfp_3$ approach $\bfp_2$ equally from opposite sides, then in the limit $\norm{\bfp_1 - \bfp_2} + \norm{\bfp_2 - \bfp_3} = \norm{\bfp_1 - \bfp_3}$, so the $s_{ijk}$ obey the relations
\begin{equation*}
1 + s_{231} = s_{132}, \quad s_{213} + 1 = s_{312}, \quad s_{123} + s_{321} = 1.
\end{equation*}

In $\cns$ the values of $\pij$ are in $S^1$ and are mapped to $S^1$ by $\cng$.
Thus, while the exact values of the unit tangent vectors $\pij$ and $\pi_{ji}$ are unknown for two colliding points on $\gamma$, they must differ by $\pi$. 

In the case of the circle, the cyclic ordering of points along $S^1$ determines $(n-1)!$ connected components of $C_n[S^1]$. We will focus on one of these connected components.

\begin{definition} 
\label{def:ccng}
Let $C_n^0[\gamma(S^1)]$ denote the component of $C_n[\gamma(S^1)]$ where the order of the points $\bfp_1, \dots, \bfp_n$ matches the cyclic order of these points along $\gamma$ according to the given parametrization of $\gamma$.
\end{definition}

 Note that some strata are empty in the boundary of each connected component of $C_n[S^1]$ (and hence $C_n[\gamma(S^1)]$). For instance, in the component of $\cfs$ where points $\bfp_1$, $\bfp_2$, $\bfp_3$ and $\bfp_4$ occur in order along $S^1$, if $\bfp_1$ and $\bfp_3$ come together, either $\bfp_2$ or $\bfp_4$ must collapse to the same point. Thus the stratum $(13)$ is empty on the boundary of this component. 

\subsection{The configuration space of square-like quadrilaterals.}\label{sect:slq}

In this section we show that $\slq$ (configurations of square-like quadrilaterals) is a submanifold of $\cfro$. Thus $\slq$ plays the role of ``$Z$'' in Section~\ref{sect:main-method}. We also show that the boundaries of $\slq$ and $C_4^0[\gamma(S^1)]$ are disjoint in $\cfr$.

\begin{definition} \label{def:slq}
We define $\slq$ to  be the subset of {\em square-like quadrilaterals} of $\cfr$ such that  $r_{124} = r_{231} = r_{342} = 1$ and $r_{132} - r_{241} = 0$.  
\end{definition}

In other words, if $\p=((\bfp_1,\bfp_2,\bfp_3,\bfp_4),\alpha(\p))\in \slq$, then the first condition implies $\norm{\bfp_1-\bfp_2} = \norm{\bfp_2-\bfp_3} = \norm{\bfp_3-\bfp_4} = \norm{\bfp_4-\bfp_1}$, and the second condition implies $\norm{\bfp_1-\bfp_3} = \norm{\bfp_2-\bfp_4}$. 
Thus when $k>2$, $\slq$ is  the space of quadrilaterals in $\R^k$ with equal sides and equal diagonals. When $k=2$, $\slq$ is the space of squares in $\R^2$. Note that we work with $r_{ijl}$ here as we are considering the actual ratios of lengths, rather than the rescaled ratios $s_{ijl}$ (see \defn{pijsijk}). 

\begin{proposition}
\label{prop:slq is submanifold of r2}
The space $\slq \cap \cfro$ is an orientable submanifold of $\cfro$.
\end{proposition}

\begin{proof}
Let $\p=((\bfp_1,\bfp_2,\bfp_3,\bfp_4),\alpha(\p))$ be a point  in $\cfr$, and consider the mapping $g\co \cfr \to \R^4$ given by
\begin{align}
\label{slq-eq} g(\p) &= (r_{124}^2, r_{231}^2, r_{342}^2, r_{132}^2 - r_{241}^2) \\ 
&= \left( \frac{\norm{\bfp_1-\bfp_2}^2}{\norm{\bfp_1 - \bfp_4}^2},\, \frac{\norm{\bfp_2-\bfp_3}^2}{\norm{\bfp_1 - \bfp_2}^2},\, \frac{\norm{\bfp_3-\bfp_4}^2}{\norm{\bfp_2 - \bfp_3}^2},\,
\frac{\norm{\bfp_1-\bfp_3}^2}{\norm{\bfp_1 - \bfp_2}^2} - \frac{\norm{\bfp_2-\bfp_4}^2}{\norm{\bfp_2 - \bfp_1}^2} \right).
\end{align}
This mapping is smooth and $\slq$ is the preimage of the point $(1,1,1,0)$. In this proof, we show
that 
$$Dg:T_{\p}\cfro\rightarrow T_{g(\p)}\R^4$$ 
is onto at points $\p\in\slq$ by showing $Dg$ has four linearly independent rows.   It then follows from the Preimage Theorem of \cite{Guillemin:2010ti} that ${g} \transverse (1,1,1,0)$ and the interior of $\slq$ is an orientable submanifold of $C_4(\R^k)$.

In order to do this, we create a basis of tangent vectors to $\cfro$ that allows $Dg$ to be computed easily. There are two cases; when $\p\in \slq$ is planar and when it is not. 
Note that we will use these bases in later proofs, so we give full details here.
We denote a tangent vector at $\p$ by $\h = \bfh(\p)=(\bfv_1,\bfv_2,\bfv_3,\bfv_4)$, where each $\bfv_i$ is a tangent vector at $\bfp_i$. (Here we suppress the $\alpha(\bfp)$ information on the strata.)  

For each of the two cases, we need to know the image of $Dg$ with respect to vectors $\h$. We compute the derivative of each  of the four equations of $g$ with respect to a vector $\h$. The details for the computation for $D_{\h}(r_{124}^2)$  is found below, where we have simplified using $\norm{\bfp_1-\bfp_2} = \norm{\bfp_2-\bfp_3} = \norm{\bfp_3-\bfp_4} = \norm{\bfp_4-\bfp_1}$ and $\norm{\bfp_1-\bfp_3}=\norm{\bfp_2-\bfp_4}$.
Thus the first row of $Dg$ contains

\begin{align*}D_{\h}(r_{124}^2) &= D_{\h}\left(\frac{\norm{\bfp_1-\bfp_2}^2}{\norm{\bfp_1-\bfp_4}^2}\right)
\\ &=
\frac{\norm{\bfp_1-\bfp_4}^22\norm{\bfp_1-\bfp_2}D_{\h}(\norm{\bfp_1-\bfp_2})- \norm{\bfp_1-\bfp_2}^22\norm{\bfp_1-\bfp_4}D_{\h}(\norm{\bfp_1-\bfp_4}) }{\norm{\bfp_1-\bfp_4}^4}
\\ & = \frac{2}{\norm{\bfp_1-\bfp_4}}\left(D_{\h}(\norm{\bfp_1-\bfp_2}) - D_{\h}(\norm{\bfp_1-\bfp_4})\right).
\end{align*}
Using similar reasoning, we compute the column of $Dg$ corresponding to a vector $\h$ to be

\begin{equation}\label{eqn:Dg}
\begin{pmatrix}
\frac{2}{\norm{\bfp_1-\bfp_4}}\left(D_{\h}(\norm{\bfp_1-\bfp_2}) - D_{\h}(\norm{\bfp_1-\bfp_4})\right)
\\ 
\frac{2}{\norm{\bfp_2-\bfp_1}}\left(D_{\h}(\norm{\bfp_2-\bfp_3}) -D_{\h} (\norm{\bfp_2-\bfp_1})\right)
\\
 \frac{2}{\norm{\bfp_3-\bfp_2}}\left(D_{\h}(\norm{\bfp_3-\bfp_4}) - D_{\h}(\norm{\bfp_3-\bfp_2})\right)
 \\ \frac{2}{\norm{\bfp_1-\bfp_2}^2} \left(\norm{\bfp_1-\bfp_3}D_{\h}(\norm{\bfp_1-\bfp_3}) - \norm{\bfp_2-\bfp_4}D_{\h}(\norm{\bfp_2-\bfp_4}) \right)
\end{pmatrix}
\end{equation}

{\bf Case 1:} \ We first assume the configuration $\p\in \slq$ is planar, then $\bfp_1\bfp_2\bfp_3\bfp_4$ are the four vertices of a square. This means the points are inscribed in an ellipse, and without loss of generality, we choose an ellipse with equation $ \nicefrac{x^2}{a^2} + \nicefrac{y^2}{b^2} = 1$ and $a > b$. (Note that for any choice of ratio $a:b$ there will be a square for precisely one pair of numbers satisfying that ratio. These numbers are determined by the sidelength of the square.)  If we parametrize the ellipse by $(x(\theta),y(\theta)) = (a \cos \theta, b \sin\theta)$, then we can show that $\cos^2 \theta = b^2/(a^2 + b^2)$ and $\sin^2\theta = a^2/(a^2 + b^2)$. 

\begin{figure}[h]
\hfill
\begin{overpic}[height=1.5in]{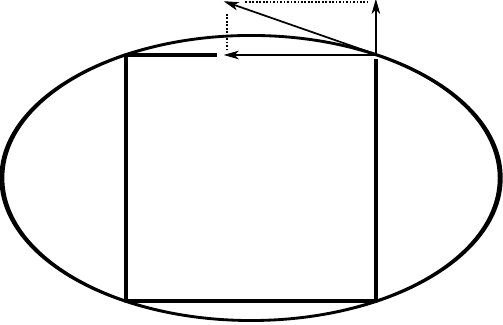}
\put(76,56){$\bfp_1$}
\put(20,56){$\bfp_2$}
\put(20,1){$\bfp_3$}
\put(76,1){$\bfp_4$}
\put(60, 60.5){$\bfv_1$}
\end{overpic}
\hfill
\hphantom{.}
\caption{
This figure shows the effect of moving only $\bfp_1 = (a \cos \theta_1,b \sin \theta_1)$ along the ellipse on the sides of the quadrilateral. This motion increases $\norm{\bfp_4 - \bfp_1}$ while decreasing $\norm{\bfp_1 - \bfp_2}$.
\label{fig:ellipsetrans}}
\end{figure}
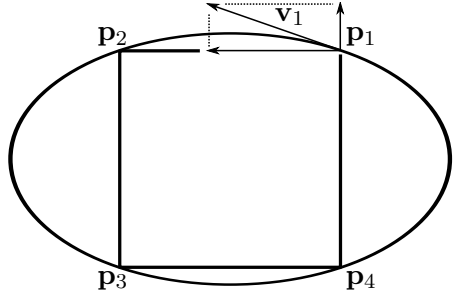

For our first tangent vector to $\cfro$, we consider the effect of moving $\bfp_1$ along the ellipse as shown in Figure~\ref{fig:ellipsetrans}. For some $\theta_1$, this point has coordinates $(a \cos \theta_1,b \sin \theta_1)=(\nicefrac{ab}{\sqrt{a^2+b^2}}, \nicefrac{ab}{\sqrt{a^2+b^2}})$.  Our first  vector is the tangent vector to the ellipse $\h_1=(\bfv_1,\bfzero,\bfzero,\bfzero)$, where $\bfv_1=(\nicefrac{-a^2}{\sqrt{a^2 + b^2}},\nicefrac{b^2}{\sqrt{a^2 + b^2}})$. Using symmetry we see the square has sidelength $\nicefrac{2ab}{\sqrt{a^2+b^2}}$ and for this vector $D_{\h_1}(\norm{\bfp_1-\bfp_2})=-\nicefrac{a^2}{a^2+b^2}$, $D_{\h_1}(\norm{\bfp_1-\bfp_4})=\nicefrac{b^2}{a^2+b^2}$ and  $D_{\h_1}(\norm{\bfp_1-\bfp_3})=\nicefrac{-a^2+b^2}{\sqrt{2(a^2+b^2)}}$ . Substituting these values in the first row of Equation~\ref{eqn:Dg}, we find
$$D_{\h_1}(r_{124}^2) = \frac{\sqrt{a^2+b^2}}{ab}\left( -\frac{a^2}{a^2+b^2}- \frac{b^2}{a^2+b^2}\right) = -\frac{a}{b} - \frac{b}{a}.
$$

We then construct vectors $\h_2$, $\h_3$ and $\h_4$ at the remaining $\bfp_i$ in an analogous fashion. (We note these vectors are linearly independent, since each of them only acts at one vertex of the square.) Using {\em Mathematica} we compute $Dg$ restricted to $\Span\{\h_1, \h_2, \h_3, \h_4\}$ to be the matrix
$$
\begin{pmatrix} 
-\frac{a}{b} - \frac{b}{a} & \frac{a}{b} & 0 & \frac{b}{a} \\
 \frac{a}{b}  & -\frac{a}{b} - \frac{b}{a} & \frac{b}{a} & 0 \\
0 & \frac{b}{a}  & -\frac{a}{b} - \frac{b}{a} & \frac{a}{b} \\
-\frac{a}{b} + \frac{b}{a} &-\frac{a}{b} + \frac{b}{a} & -\frac{a}{b} + \frac{b}{a}  & -\frac{a}{b} + \frac{b}{a}
\end{pmatrix}.
$$
This matrix has determinant $\displaystyle \frac{8(a^4-b^4)}{a^2b^2}$ which is positive, since we assumed $a> b$. This also rechecks our previous observation that $Dg$ has four linearly independent rows.

{\bf Case 2:}  We next make a similar construction for nonplanar configurations in $\slq$. Assume the square-like quadrilateral $\p\in\slq$ has sides of length $\ell=\norm{\bfp_1-\bfp_2}=\norm{\bfp_3-\bfp_2} = \norm{\bfp_4-\bfp_3}= \norm{\bfp_1-\bfp_4}$, and diagonals have length $m= \norm{\bfp_1-\bfp_3}= \norm{\bfp_2-\bfp_4}$. 

Figure~\ref{fig:slqmotions2} shows the construction of two types of tangent vectors to $\cfro$ at $\p$. The first three tangent vectors are of the form $\h_1 = (\bfv_1,\bfzero,\bfzero,\bfzero)$, $\h_2 = (\bfzero,\bfzero,\bfv_3,\bfzero)$ and $\h_3 = (\bfzero,\bfzero,\bfzero,\bfv_4)$. Vector $\h_2$ is shown on the left of Figure~\ref{fig:slqmotions2}, and observe that $\bfv_3$ is the vector at $\bfp_3$ perpendicular to the plane through $\bfp_1\bfp_3\bfp_4$. This means the directional derivatives of $\norm{\bfp_3-\bfp_4}$ and $\norm{\bfp_3-\bfp_1}$ in the direction $\h_2$ are zero (to first order). On the other hand, since the quadrilateral  is nonplanar, edge $\bfp_2\bfp_3$ is not in the plane normal to $\bfv_3$, so the directional derivative of $\norm{\bfp_2-\bfp_3}$ is nonzero.

\begin{figure}
\hfill
\begin{overpic}[height=1.5in]{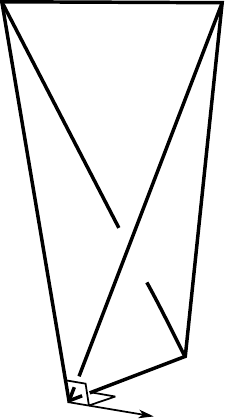}
\put(-4,103){$\bfp_4$}
\put(47,10){$\bfp_1$}
\put(55,102){$\bfp_2$}
\put(12,-3){$\bfp_3$}
\put(25, -5){$\bfv_3$}
\end{overpic}
\hfill
\begin{overpic}[height=1.5in]{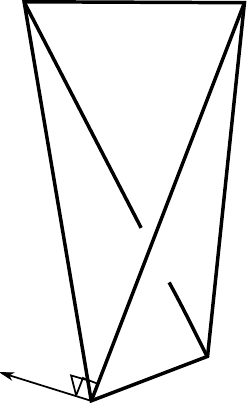}
\put(0,103){$\bfp_4$}
\put(52,5.5){$\bfp_1$}
\put(62,102){$\bfp_2$}
\put(14,-5){$\bfp_3$}
\put(0,-2){$\bfw$}
\end{overpic}
\hfill
\hphantom{.}
\caption{Two tangent vectors to a configuration $\p\in\slq$ which forms a nonplanar quadrilateral. On the left, we see the tangent vector where the directional derivative of $\norm{\bfp_2-\bfp_3}$ is positive, while the directional derivatives of all other lengths vanish. A similar tangent vector may be constructed at vertices $\bfp_1$ and $\bfp_4$. On the right, we see the tangent vector  where the directional derivative of $ \norm{\bfp_1 - \bfp_3}$ is positive while the directional derivatives of all other lengths vanish.
}
\label{fig:slqmotions2}
\end{figure}
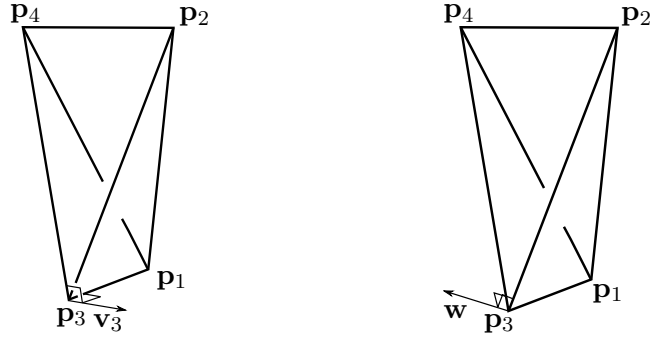

We define vectors $\h_1$ and $\h_3$ in  a similar way. Vector $\bfv_1$ is perpendicular to the $\bfp_1\bfp_2\bfp_3$ plane at $\bfp_1$, and vector $\bfv_4$ is perpendicular to the $\bfp_1\bfp_2\bfp_4$ plane at $\bfp_4$. To summarize, we choose $\bfv_1, \bfv_3, \bfv_4$ to be scaled so that
\begin{align*}
D_{\h_1} \norm{\bfp_1 - \bfp_4} &= -\ell/2, \quad \text{other directional deriv.'s of lengths } = 0, \\
D_{\h_2} \norm{\bfp_2 - \bfp_3} &= \ell/2, \quad \text{other directional deriv.'s of lengths } = 0, \\
D_{\h_3} \norm{\bfp_3 - \bfp_4} &= \ell/2, \quad \text{other directional deriv.'s of lengths } = 0.
\end{align*}
The fourth tangent vector $\h_4=(\bfzero,\bfzero,\bfw,\bfzero)$ is shown at the right of Figure~\ref{fig:slqmotions2}. Vector $\bfw$ is perpendicular to the plane through $\bfp_2\bfp_3\bfp_4$ at $\bfp_3$. We can choose $\bfw$ to be scaled so that
$$
D_{\h_4} \norm{\bfp_1 - \bfp_3} = \ell^2/2m, \quad \text{other directional deriv.'s of lengths } = 0. 
$$
Substituting these values in Equation~\ref{eqn:Dg}, we see that various factors of $\ell$ and $m$ cancel, and $Dg$ restricted to $\Span\{\h_1, \h_2, \h_3, \h_4\}$ is the matrix:
\begin{equation*}
Dg = \begin{pmatrix}
1  & 0  & 0 & 0 \\
0  & 1  & 0 & 0 \\
0  & -1 & 1 & 0 \\
0  & 0  & 0 & 1 \\
\end{pmatrix}.
\end{equation*}
This matrix has determinant 1, and again we see that $Dg$ has four linearly independent rows. 

In both cases ($\slq$ planar and non-planar), we have shown that $Dg$ has four linearly independent rows and is onto. This means that $g\pitchfork (1,1,1,0)$, and the interior of $\slq$ is an orientable submanifold of $\cfro$.
\end{proof} 

We now discuss the behavior of the boundaries of $\slq$ and $ C_4^0[\gamma(S^1)]$. Recall that, a point lies in $\bdry C_4[\R^k]$ when points of a configuration come together, along with the directions of collision and ratios of the sides. 

\begin{lemma} \label{lem:boundry-disjoint} The spaces $\slq$ and $C_4^0[\gamma(S^1)]$ are boundary-disjoint in $\cfr$.
\end{lemma}

\begin{proof}

The boundary of $\slq$ can occur in two ways: when both the sidelengths and diagonals vanish giving an ``infinitesimal'' square-like quadrilateral, and when the diagonals vanish, but the sidelengths remain equal and nonzero. That is, $\bdry \slq$ lies in the boundary faces $(1234)$ and $(13)(24)$ respectively of $\bdry\cfr$. We observe that when $k=2$, $\slq$ has two connected components: one is the interior of $\slq$ and the $(1234)$ boundary face, the other is the $(13)(24)$ boundary face. (The latter is disjoint from the interior of $\slq$: in the plane, there is no way for the diagonals to vanish without the sidelengths vanishing too.) When $k>2$, $\slq$ is connected.

As we observed in at the end of Section~\ref{sect:cfg},  $C_4^0[\gamma(S^1)]$ does not contact the $(13)(24)$ or $((13)(24))$ faces of $\cfr$. Thus we need only consider the portion of $\slq$ and $C_4^0[\gamma(S^1)]$ on the interior of the $(1234)$ face. For $\slq$, these configurations are infinitesimal tetrahedra with equal sides and equal diagonals. Thus each configuration $\bfp_1\bfp_2\bfp_3\bfp_4$ contains four congruent triangles, each of which has one diagonal and two edges. This means the four internal angles $\angle \bfp_{i-1}\bfp_{i}\bfp_{i+1} \pmod 4$ are all equal. (The remaining eight angles in the four triangles are also equal to one another.) For $C_4^0[\gamma(S^1)]$, since $\gamma$ is smooth, these configurations are an infinitesimal collinear quadrilateral with internal angles of either $0$ or $\pi$. Thus the $\pi_{ij}$ and $s_{ijk}$ data are completely different for $\bdy\slq$ and $\bdry C_4^0[\gamma(S^1)]$, and the spaces are are boundary disjoint.
\end{proof}

In Appendix~\ref{sect:bdry-slq}, we give further details about the structure of the faces of $\bdry\slq$. There, in  Proposition~\ref{prop:slq-boundary}, we show that each of the boundary $(1234)$ and $(13)(24)$ faces of $\slq$ is a submanifold of $\cfr$.  The proofs of these results are similar in flavor to the proof of \prop{slq is submanifold of r2}, so are omitted here. Note that we will refer to Proposition~\ref{prop:slq-boundary}  later in the proof of Proposition~\ref{prop:action}. 

\subsection{Cyclic actions}\label{sect:cyclic}
The next step on our method is to consider a base case. For us this will be the ellipse. Step~1 of our method requires us
to compute the homology class in $H_0(\slq,\Z)$ of the intersection of $C_4^0[\gamma(S^1)]$ and $\slq$ for a transverse intersection. Unfortunately for us, while $\slq \cap C_4^0[\gamma(S^1)]$ is indeed $0$-dimensional, the intersection represents $0$ in the homology $H_0(\slq;\Z) = \Z$. The essential problem is that a square-like quadrilateral can be cyclically relabeled in four ways, and it turns out that these relabelings alternate signs in $H_0(\slq;\Z)$. Indeed H.B.~Griffiths~\cite{MR1095236} took a similar approach, though he seems to have failed to appreciate the orientation-reversing nature of the cyclic permutation on $\cfr$. As a result, he (wrongly) computes a different intersection number to be 16 instead of zero. Fortunately, we can fix the problem by identifying these relabelings as a single configuration.  

\begin{definition} \label{def:cyclic}
Let $\mu \co \cfr \rightarrow \cfr$ be the map corresponding to the generator of $\Z/4\Z$ for the action on $\cfr$ that cyclically permutes $\bfp_1$, $\bfp_2$, $\bfp_3$ and $\bfp_4$, namely
$$ \mu(\bfp_1, \bfp_2,\bfp_3,\bfp_4)=(\bfp_2, \bfp_3,\bfp_4, \bfp_1).
$$ 
It is clear from the definition of $\slq$ that $\mu$ descends to a map from $\slq$ to $\slq$.
\end{definition}

This action is free\footnote{The action is also automatically properly discontinuous as the group is finite.}. To see this, simply repeat the arguments Sinha  uses for the symmetric group (cf. Theorem~4.10 of~\cite{newkey119}).

\begin{lemma}\label{lem:cfr-orient}
The map $\mu$ reverses orientation on $\cfr$ if $k$ is odd, and preserves orientation if $k$ is even.
\label{lem:orientation}
\end{lemma}

\begin{proof} Observe that the tangent space to $\cfr$ contains of four copies of $T\R^k$ and that reordering these from $(1,2,3,4)$ to $(2,3,4,1)$ requires $3k^2$ swaps of basis elements. Thus $\mu$ is orientation reversing or  preserving on $\cfr$ as $k$ is odd or even. 
\end{proof}

In order to understand how the action $\mu$ affects the orientation on $\slq=g^{-1}(1,1,1,0)$, note  that $\slq$ carries the preimage orientation.  We now review the definition of the  preimage orientation from \cite{Guillemin:2010ti}. 
Assume $f \co X \rightarrow Y$ is transverse to $Z \subset Y$, let $f(x)=z\in Z$, and suppose $H_x$ is a  subspace of $T_xX$ complementary to the the subspace $T_x(f^{-1}(Z))$. Then the orientation of $Z$ and $Y$ induce a direct image orientation on $Df_x H_x$.  Since  $T_x(f^{-1}(Z))$ contains the entire kernel of $Df_x$, then $Df_x$ maps $H_x$ isomorphically onto its image. The induced orientation on $Df_xH_x$ defines an orientation on $H_x$ via the map $Df_x$. In summary, the two direct sums
\begin{align*} Df_xH_x\oplus T_z(Z) & = T_z(Y),
\\ H_x\oplus T_x(f^{-1}(Z)) &= T_x(X) 
\end{align*}
define the orientation on $H_x$ and on $T_x(f^{-1}(Z))$.


\begin{proposition} \label{prop:orientation}
The map $\mu$ reverses orientation on $\slq\cap\cfro$ if $k$ is odd, and preserves orientation if $k$ is even.
\end{proposition}

\begin{proof} Since $\slq=g^{-1}(1,1,1,0)$, where $g$ is given by Equation~\ref{slq-eq}, 
then $T_{\p}(\slq)\subset T_{\p}(\cfr)$. Recall from the proof of \prop{slq is submanifold of r2}, that a tangent vector at $\p$ is denoted by $\h = \bfh(\p)=(\bfv_1,\bfv_2,\bfv_3,\bfv_4)$, where $\bfv_i$ is a tangent vector at $\bfp_i$ and the $\alpha(\p)$ information has been suppressed.
  In the following argument, we will need the two direct sums from the definition of preimage orientation translated to our setting. Assuming that $g(\p)=(1,1,1,0)$ and noting that $T_{(1,1,1,0)}(1,1,1,0)=\{0\}$, then we obtain
\begin{align} Dg_{\p} H_{\p}\oplus \{0\} & = T_{(1,1,1,0)}(\R^4),      \label{PO1}
\\ H_{\p} \oplus T_{\p}(\slq) &= T_{\p}(\cfr).        \label{PO2}
\end{align}

To prove the proposition, use the variations of quadrilaterals in $\slq$ seen in the proof of \prop{slq is submanifold of r2}. These turn out to behave nicely under the $\Z/4\Z$ action! There are two cases, depending on whether the point $\p\in\slq$ is a planar square, or is a nonplanar quadrilateral. In both cases we define $H_{\p}=\Span\{\h_1, \h_2,\h_3,\h_4\}$ where the $\h_i$'s were defined in the proof of \prop{slq is submanifold of r2}.
In each case,  we proved the vectors are linearly independent, thus give a basis $\mathcal{B}_H$ for $H_{\p}$, which is the  subspace of $T_{\p}(\cfr)$ complementary to $T_{\p}(\slq)$.

In the proof of \prop{slq is submanifold of r2}, we saw that for both cases $Dg$ has positive determinant so is an orientation preserving isomorphism onto its image. Since $\R^4$ and $T_{(1,1,1,0)}(\R^4)$ have the standard orientation,  we deduce that our bases $\mathcal{B}_H$ for $H$ are positively oriented using Equation~\ref{PO1}.  Note that $\cfr$ inherits a positive orientation from picking a consistent positive orientation on each $\R^k$. Using Equation~\ref{PO2}, we can define a basis $\mathcal{B}_S$ of $T_{\p}(\slq)$ which gives a positive orientation. This orientation is the preimage orientation of $T_{\p}(\slq)$.
 
The map $\mu:\cfr\rightarrow \cfr$ is a diffeomorphism, so is an isomorphism on each tangent space. 
We can check whether this isomorphism is orientation preserving or reversing on $\slq$ by another computation.  
Now, $\mu$ restricts to a map $\mu:\slq\rightarrow \slq$, and hence $D\mu_{\p}:T_{\p}\slq\rightarrow T_{\mu(\p)}\slq$. Applying $D\mu_{\p}$ to each side of Equation~\ref{PO2} gives
\begin{equation}
D\mu_{\p}(H_{\p}) \oplus D\mu_{\p}(T_{\p}(\slq)) \cong D\mu_{\p}(T_{\p}(\cfr))\cong T_{\mu(\p)}(\cfr).\label{PO3}
\end{equation}

\lem{cfr-orient} showed that the last isomorphism is orientation preserving if $k$ is even, and orientation reversing if $k$ is odd.
We computed above that $\mathcal{B}_H$ is a positively oriented basis of the complementary space $H_{\p}$. We need to know if these vectors push-forward to a positively oriented basis for the subspace $D\mu_{\p}(H_{\p})$. To do this, observe that Equations~\ref{PO1} and \ref{PO2} hold for any complementary subspace of $T_{\p}(\slq)$ and any point $\p\in\slq$, in particular for $D\mu_{\p} (H_{\p})$ and $\mu(\p)$.
Thus
\begin{align} Dg(D\mu_{\p} (H_{\p}))\oplus \{0\} & = T_{(1,1,1,0)}(\R^4),  \label{PO4}
\\ D\mu_{\p}(H_{\p}) \oplus T_{\mu(\p)}\slq & = T_{\mu(\p)}(\cfr). \label{PO5}
\end{align}
We thus check what happens to $\mathcal{B}_H$ under $Dg$. As before, there are two cases, depending on whether $\p\in\slq$ is planar or non-planar. We will see below that in both cases $\det Dg>0$, and so $Dg$ is orientation preserving. Combining information from Equations~\ref{PO3} and \ref{PO5} leads us to conclude the following. When $k$ is even, the orientation of $ D\mu_{\p}(T_{\p}(\cfr))$ matches that of $T_{\mu(\p)}(\cfr)$, and so we deduce that the push-forward orientation on $D\mu_{\p}(T_{\p}\slq)$ is equal to the preimage orientation on $T_{\mu(\p)}\slq$ as claimed. When $k$ is odd, the push-forward orientation of $ D\mu_{\p}(T_{\p}(\cfr))$ is opposite that of $T_{\mu(\p)}(\cfr)$, and thus the orientation on $D\mu_{\p}(T_{\p}\slq)$ is opposite to the preimage orientation on $T_{\mu(\p)}\slq$, also as claimed. Note that we need both cases, because when $k>2$, we can still have planar squares in $\slq$ (and our other variational vector fields will be harder to define).

All that remains to complete the proof, is to show that $\det Dg>0$ and so $Dg$ is orientation preserving.  {\bf Case 1:} Assume that $\p\in\slq$ is a planar configuration, and use the corresponding  basis $\{\h_1, \h_2, \h_3,\h_4\}$  of vectors moving points along the ellipse from the proof of \prop{slq is submanifold of r2}. 
To help keep track of the action of $\mu$, denote $\mu(\p)=(\bfp_2, \bfp_3,\bfp_4,\bfp_1)=(\hat{\bfp}_1,\hat{\bfp}_2,\hat{\bfp}_3, \hat{\bfp}_4)$, and so $\hat{\bfp}_1=\bfp_2,\dots, \hat{\bfp}_4=\bfp_1$. 
Note that $D\mu_{\p}(\h_1)=D\mu_{\p}\begin{bmatrix} \bfv_1\\0\\0\\0\end{bmatrix} = \begin{bmatrix} 0\\0\\0\\ \bfv_1\end{bmatrix}$, and so the vector $\bfv_1$ is at point $\hat{\bfp}_4=\bfp_1$. Thus $D\mu_{\p}(\h_1)\norm{\hat{\bfp}_4-\hat{\bfp}_3}=\frac{b^2}{\sqrt{a^2+b^2}}$, $D\mu_{\p}(\h_1)\norm{\hat{\bfp}_1-\hat{\bfp}_4}=-\frac{a^2}{\sqrt{a^2+b^2}}$, and $D\mu_{\p}(\h_1)\norm{\hat{\bfp}_4-\hat{\bfp}_2}=\nicefrac{-a^2+b^2}{\sqrt{2(a^2+b^2)}}$. Computing using the first row of Equation~\ref{eqn:Dg} gives
$$D\mu_{\p}(\h_1) (r_{124}^2) = \frac{\sqrt{a^2+b^2}}{ab}\left(0+ \frac{a^2}{a^2+b^2}\right) = \frac{a}{b}.$$
Once again, we use {\em Mathematica} and compute that on the space $\Span\{D\mu_{\p}(\h_1), D\mu_{\p}(\h_2), D\mu_{\p}(\h_3), D\mu_{\p}(\h_4)\}$, we have:
$$
Dg=\begin{pmatrix} 
\frac{a}{b} &-\frac{a}{b} - \frac{b}{a} & \frac{b}{a} & 0  \\
0 &  \frac{b}{a}  & -\frac{a}{b} - \frac{b}{a} & \frac{a}{b}  \\
 \frac{b}{a} & 0   & \frac{a}{b}  & - \frac{a}{b}- \frac{b}{a}  \\
\frac{a}{b} - \frac{b}{a} & \frac{a}{b} - \frac{b}{a} & \frac{a}{b} - \frac{b}{a}  & \frac{a}{b} - \frac{b}{a}
\end{pmatrix}.
$$
This matrix has determinant $\displaystyle \frac{8(a^4-b^4)}{a^2b^2}$ which is positive (since we assumed $a>b$). 

{\bf Case 2:} Assume that $\p\in\slq$ corresponds to a nonplanar configuration, and use the other basis (also named) $\{\h_1, \h_2, \h_3,\h_4\}$ from the proof of \prop{slq is submanifold of r2}. As before denote $\mu(\p)=(\bfp_2, \bfp_3,\bfp_4,\bfp_1)=(\hat{\bfp}_1,\hat{\bfp}_2,\hat{\bfp}_3, \hat{\bfp}_4)$. By using similar reasoning, we observe that $D\mu_{\p}(\h_1)\norm{\hat{\bfp}_4-\hat{\bfp}_3}=-\ell/2$, and all other derivatives are 0. Thus all entries of the first column of $Dg$ are zero, except for
$$D\mu_{\p}(\h_1) (r_{342}^2) = \frac{2}{\ell}\left(-\frac{\ell}{2}-0 \right) =-1.$$

Once again, we compute that on the space $\Span\{D\mu_{\p}(\h_1), D\mu_{\p}(\h_2), D\mu_{\p}(\h_3), D\mu_{\p}(\h_4)\}$, we have:
\begin{equation*}
Dg = \begin{pmatrix}
0  &  1  & 0 & 0 \\
0   &  -1  &  1 & 0 \\
-1   &  0  &  -1 & 0 \\
0   &  0 &  0 & -1 
\end{pmatrix},
\end{equation*}
and $\det Dg=1>0$, as desired. 
\end{proof}

\begin{remark} Note that \prop{orientation} can be extended to hold for the $(1234)$ face of $\bdry\slq$. However, the $(13)(24)$ face needs an additional computation. We have chosen not to include this computation, since $\bdry\slq$ and $\bdry C_4^0[\gamma(S^1)]$ are disjoint in $\cfr$, and we don't need it for the following result.
\end{remark}

\begin{proposition}\label{prop:action}
The manifolds $\cfr$, $C_4^0[\gamma(S^1)]$, and $\slq$ share a smooth, free, and properly discontinuous $\Z/4\Z$ action given by cyclically relabeling points in a configuration. 
\begin{compactenum}
\item The generator $(\bfp_1,\bfp_2,\bfp_3,\bfp_4) \mapsto (\bfp_2,\bfp_3,\bfp_4,\bfp_1)$ is always orientation-reversing on $\cfog$. It is orientation-reversing on both $\cfr$ and $\slq$ if $k$ is odd, and orientation preserving on $\cfr$ and $\slq$ if $k$ is even.
\item The quotient spaces by the action of $\Z/4\Z$, $\qcfr$ and $\hat{C}_4^0[\gamma(S^1)]$, are manifolds-with-boundary and corners, with $\hat{C}_4^0[\gamma(S^1)]$ non-orientable. Also, $\qcfr$ is non-orientable when $k$ is odd, and orientable when $k$ is  even.
\item The intersection of $\qslq$ (the quotient space by the action of $\Z/4\Z$) with the complement of an $\epsilon$-neighborhood of the boundary face $(13)(24)$ (which is preserved under the action), is a manifold-with-boundary. It is orientable precisely when $\qcfr$ is.
\item The spaces $\qslq$ and $\hat{C}_4^0[\gamma(S^1)]$ are boundary-disjoint in $\qcfr$.
\end{compactenum}
\label{prop:quotient}
\end{proposition}

\begin{proof}
We have already seen that the action on $\cfr$ is smooth, free and properly discontinuous and that it descends to a corresponding action on the submanifolds $C_4^0[\gamma(S^1)]$ and $\slq$.  It is straightforward to see this action is orientation-reversing on $C_4^0[\gamma(S^1)]$. \lem{cfr-orient} and \prop{orientation} proves the last part of (1). Statement (2) follows immediately. Statement (3) follows from Propositions~\ref{prop:slq is submanifold of r2}, \ref{prop:slq-boundary}, and~\ref{prop:orientation}. 
We note for (4) that the action is actually an isometry on $\cfg$, so it does indeed descend to the $\epsilon$-neighborhood of $(13)(24)$. Thus the quotient spaces remain boundary-disjoint.
\end{proof}

In conclusion, the spaces we will apply our method to from Section~\ref{sect:main-method} are $\qcfr$, $Z=\qslq$ and $\hat{C}_4^0[\gamma(S^1)]$.


\subsection{Base case and conclusion}
\label{sect:basecase}
Before we complete our arguments, we need to consider a base case. For us this is a planar ellipse.
\begin{lemma}
In $\R^2$, if the image of $\gamma\co S^1\rightarrow \R^2$ is a planar ellipse $ \nicefrac{x^2}{a^2} + \nicefrac{y^2}{b^2} = 1$ with $a>b$, then $\hat{C}^0_4[\gamma(S^1)] \cap \qslq\neq\emptyset$, and the intersection represents a single square.
\end{lemma}

\begin{proof} We prove that the intersection is a single square. This is a straightforward computation relying on the symmetry of the ellipse and is found in  \lem{one-square-ellipse} in Appendix~\ref{sect:ellipse}.
\end{proof}

\begin{proposition}\label{prop:ellipse-transverse}
In $\R^2$,  if the image of $\gamma\co S^1\rightarrow \R^2$  is a planar ellipse $ \nicefrac{x^2}{a^2} + \nicefrac{y^2}{b^2} = 1$ with $a>b$, then $\hat{C}^0_4[\gamma(S^1)]$ and $\qslq$ intersect transversely (namely, $\qcfog \transverse \qslq$), and the intersection represents a single square.
\end{proposition}

\begin{proof} Recall that $\slq$ is the preimage of $(1,1,1,0)$ under the map $g$ given by Equation~\ref{slq-eq}. To prove that  $C_4^0[\gamma(S^1)]$ is transverse to  $\slq$, we show that $g$ restricted to $C_4^0[\gamma(S^1)]$ is transverse to $(1,1,1,0)$. Recall from the proof of \prop{slq is submanifold of r2} that in Case 1, we considered the planar case. We started with a square inscribed in an ellipse. We then considered tangent vectors corresponding to motions of the square along (and tangent to) the ellipse. We computed $Dg$ restricted to these vectors. Since $Dg$ had nonzero determinant, this means we proved that $C_4^0[\gamma]\pitchfork \slq$. 

Recall that the quotient spaces $\hat{C}^0_4[\gamma(S^1)]$ and $\qslq$ arise from the action of the map $\mu$ (see  \prop{quotient}) that cyclically permutes the coordinates. The map $\mu$ is differentiable and an isometry. If we look at the intersection of the quotient spaces $\hat{C}^0_4[\gamma(S^1)]$ and $\qslq$, then this is isometric to any of the 4 pre-images. Since transversality is a local computation, it descends to the quotient spaces. We can then conclude that the quotient spaces $\hat{C}_4^0[\gamma(S^1)]$ and $\qslq$ are transverse as well. 
\end{proof}

As we remarked at the beginning of Section~\ref{sect:cyclic}, the number of labeled squares must be even because every square is counted 4 times. So the homology class of $C_4^0[\gamma(S^1)]\cap\slq$  in $H_0(\slq;\Z/2\Z)$ is zero. However, taking quotients mod $\Z/4\Z$ fixes this problem, and we conclude the following.

\begin{corollary} \label{cor:ellipse-homology}
The homology class of $\hat{C}_4^0[\gamma(S^1)]\cap\qslq$  in $H_0(\qslq;\Z/2\Z)$ is 1. 
\end{corollary} 

We are now ready to prove our version of the square-peg theorem. 

\begin{theorem} 
Take any regular, $C^\infty$-smooth embedding of a curve $\gamma\co S^1\hookrightarrow \R^k$. Then for all $\epsilon >0$, there is a $C^\infty$-open neighborhood of $\gamma$, in which there is, for all $m$, a $C^m$-dense set of smooth embeddings $\gamma'\co S^l\hookrightarrow \R^k$, 
 with $\|\hat{C}_4^0[\gamma'(S^1)]-\hat{C}_4^0[\gamma(S^1)]\|_0<\epsilon$, and $\hat{C}^0_4[\gamma'] \transverse\qslq$.  
Moreover,
$$
\hat{C}_4^0[\gamma'(S^1)] \cap \qslq = \{ \text{an odd, finite set of inscribed square-like quadrilaterals} \}.
$$
\label{thm:squarepeg}
\end{theorem}

As a reminder to the reader, we note that the $C^m$-dense set of smooth embeddings is with respect to the Whitney $C^\infty$-topology from Section~\ref{sect:transversality}.
 
\begin{proof}
We follow the method outlined in Section~\ref{sect:main-method} for \thm{main-theorem}.
\begin{compactenum}
\item[Step 0:] For any smooth embedding $\gamma\co S^1\hookrightarrow \R^k$, \prop{quotient} guarantees  $\hat{C}_4^0[\gamma(S^1)]$ is a submanifold of $\qcfr$. 

\item[Step 1:]  Propositions~\ref{prop:slq is submanifold of r2},~\ref{prop:orientation} and~\ref{prop:action} guarantee 
 $\qslq\cap\qcfro$ is a submanifold of $\qcfro$ with $\bdry\qslq\subset \bdry\qcfr$. 
\lem{boundry-disjoint} and \prop{quotient} show that $\hat{C}_4[\gamma(S^1)]$ and $\qslq$ are boundary-disjoint.

\item[Step 2:] Our standard embedding $i\co S^1\hookrightarrow \R^k$ is an ellipse. \prop{ellipse-transverse}, establishes the existence of a transverse intersection between $\hat{C}_4[i(S^1)]$ and $\qslq$ inside $\qcfro$. Corollary~\ref{cor:ellipse-homology} shows the homology class of the intersection is 1 in $H_0(\qslq;\Z/2\Z)$.

\item[Step 3:] Fix an $\epsilon>0$.  We can adjust the proof of our transversality theorem (Corollary~\ref{cor:transversality}) to apply to our setting. This gives us a $C^\infty$-open neighborhood of $\gamma$ in which there is, for all $m$, a $C^m$-dense set of  smooth embeddings $\gamma':S^1\hookrightarrow \R^k$, such that $\|\hat{C}[\gamma'] - \hat{C}_4[\gamma]\|<\epsilon$, and $\hat{C}^0_4[\gamma'] \transverse \qslq$, and for which $\bdry \qslq$  and $\bdry\hat{C}^0_4[\gamma'(S^1)]$ are disjoint in $\hat{C}^0_4[\R^k]$.

\item[Step 4:] We can adjust the proofs of \thm{isotopy} and \thm{homology} to our setting. This lets us conclude that the intersections $\hat{C}^0_4[i(S^1)] \cap \qslq$ and $\hat{C}^0_4[\gamma'(S^1)] \cap \qslq$  represent the same homology class in $\qslq$. 
\end{compactenum}

This means that the finite collection of points (0-manifold) $\qcfog \cap \qslq$ is cobordant by a 1-manifold to the single square in the initial ellipse in $\qslq$, and hence that the number of inscribed squares is odd.
\end{proof}

\begin{figure}
\hfill
\includegraphics[width=1.25in]{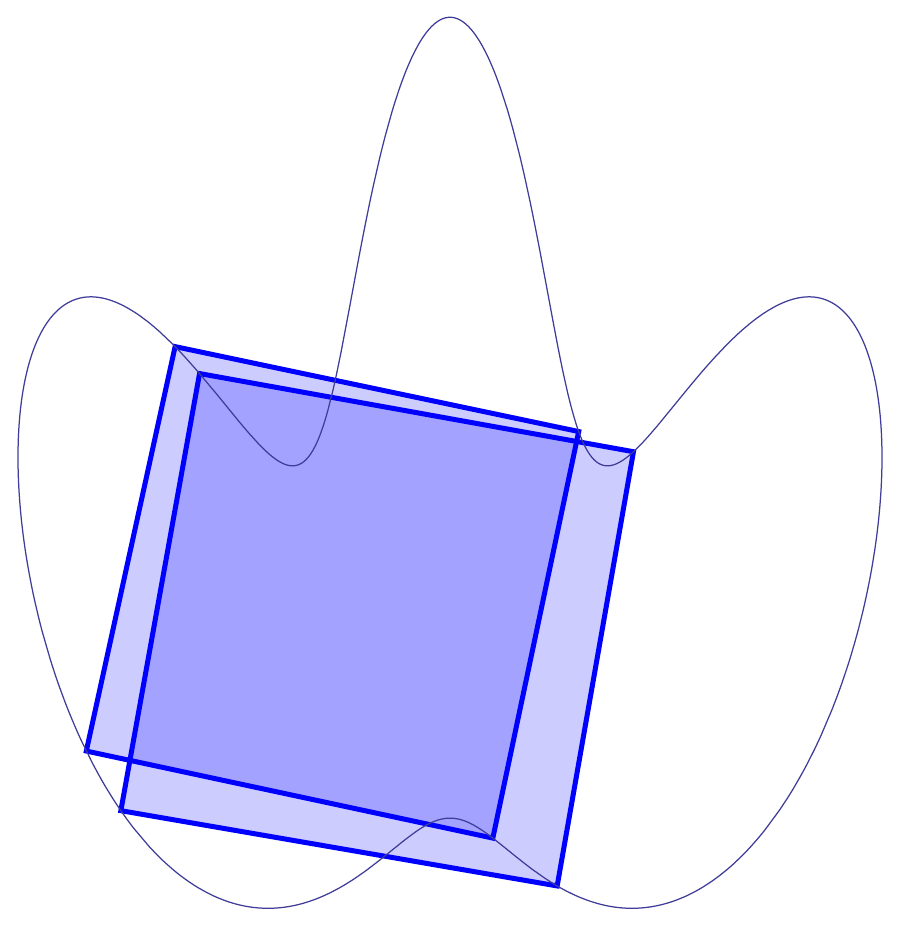}
\hfill
\includegraphics[width=1.25in]{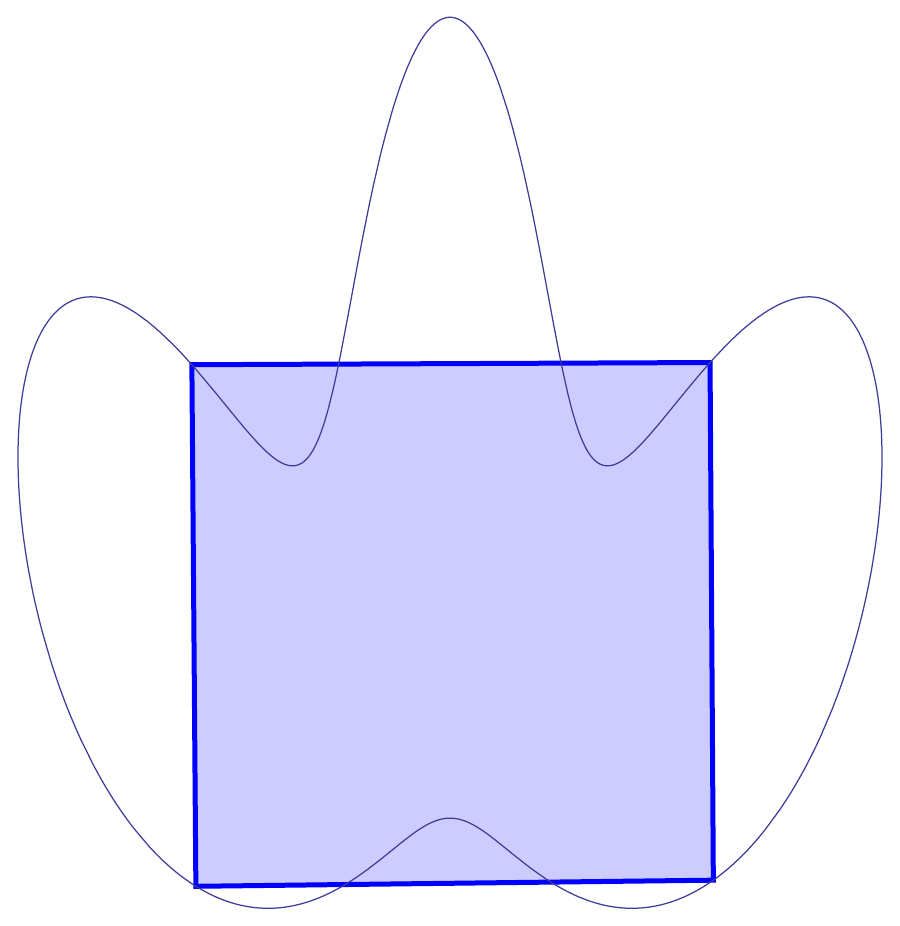}
\hfill
\includegraphics[width=1.25in]{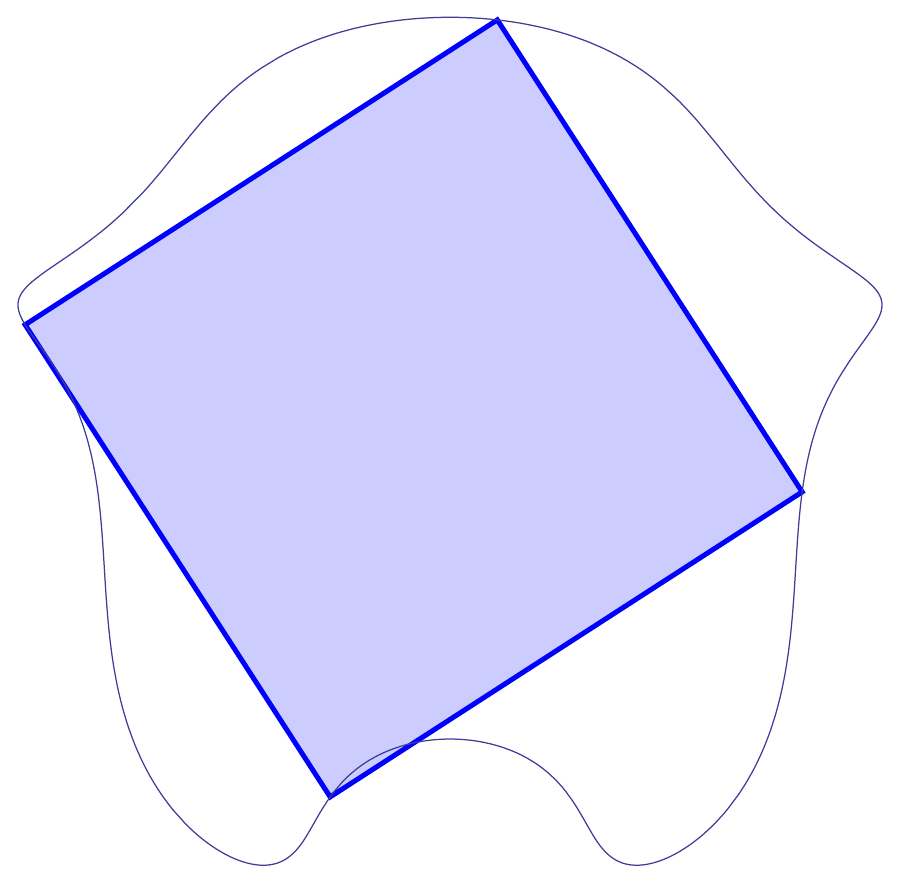}
\hfill
\includegraphics[width=1.25in]{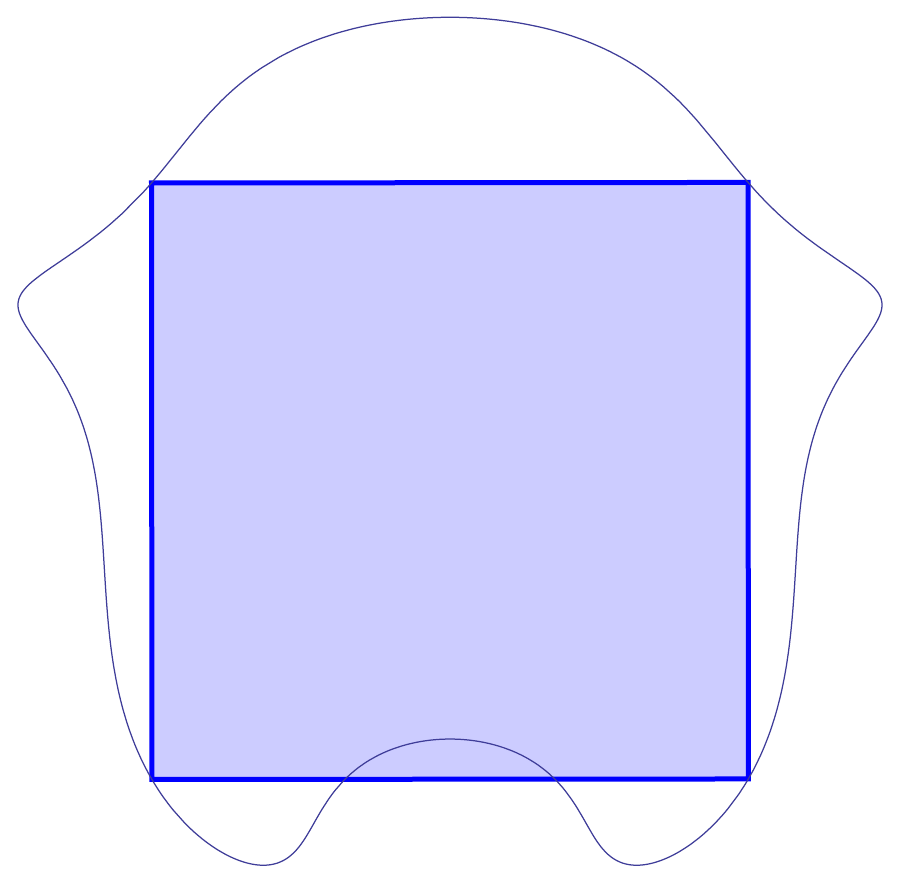}
\hfill
\hphantom{.}
\caption{This picture shows three of the five squares inscribed in an irregular three-lobed curve and two of the three squares inscribed in an irregular ``tooth-shaped'' curve. Since each family shares the vertical flip symmetry of each curve, we show the center (symmetric) square in the second and fourth pictures, while the first and third show half of the asymmetrical squares. While on the left curve the squares are fairly close together, a computer search reveals that they are certainly distinct.  
\label{fig:slq}}
\end{figure}

\thm{squarepeg} is illustrated by the 3 squares inscribed in an irregular curve shown on the left and center-left in Figure~\ref{fig:slq}. This curve has 5 squares in total (an additional 2 squares can be found using a vertical flip symmetry). The curve on the right and center right in  Figure~\ref{fig:slq} has 3 inscribed squares in total. We note that when $\cfog$ is not transverse to $\slq$ this count need not be odd. Indeed Popvassiliev~\cite{2008arXiv0810.4806P}, and F.~Sagols and R.~Mar\'in~\cite{MR2798015} have constructed, respectively, smooth convex curves, and piecewise linear curves which admit exactly $n$ inscribed squares. In addition, W.~van Heijst \cite{vanheijst2014algebraic} proved that any real algebraic curve of degree $n$ in $\R^2$ inscribes either infinitely many squares or at most $\nicefrac{n^4-5m^2+4m}{4}$ squares.

A few historical comments are in order here. First, this is certainly not the first proof of the square-peg theorem to use an intersection-theoretic approach. As previously mentioned Griffiths~\cite{MR1095236} took a similar approach, though he (wrongly) computes the intersection number. As a result, he claims to have proved not only the square-peg theorem but a ``rectangular-peg theorem''. The rectangular case does not admit the quotient-space simplification above (there are generally \emph{two} inscribed rectangles of a given aspect ratio in the ellipse). For a long while, the ``rectangular-peg theorem'' proved to be an open and difficult problem. However, this was recently solved by J.E.~Greene and A.~Lobb in \cite{MR4298749}. 
We also note that Matschke~\cite{2010arXiv1001.0186M} proved a version of the square-peg theorem from a theorem about loops of polygons inscribed in curves by arguing that a loop of rhombi which was invariant under the cyclic permutation contained a square by the intermediate value theorem, also an approach followed by L.G.~Schnirel'man~\cite{Schnirelman:1944wf}.  Additionally, in his PhD thesis \cite{phd-Matschke} Matschke claims a similar argument shows any regular $C^\infty$-smooth embedding of $S^1$ in $\R^k$ has an inscribed square-like quadrilateral.


\section{Future Directions}

We have already mentioned that we have applied the method given in Section~\ref{sect:main-method} to the question of inscribing constructible simplices in embedded spheres in \cite{Simplices}. These results generalize the results of M.D.~Meyerson~\cite{MR600575}, M.~Nielsen~\cite{MR1181760}, and others \cite{gupta2021inscribed, phd-Matschke}  on inscribing families of triangles in planar and spatial curves.

One of the recurring features of the method in this paper is that the introduction of compactified configuration spaces simplifies many of the tricky technical pieces in the proof by exporting the troublesome behavior to the boundaries.  For example, applying a transversality theorem to squares and configurations of inscribed quadrilaterals requires us to have some strategy for dealing with ``degenerate'' configurations. The extension of the $\pi_{ij}$ and $s_{ijk}$ data to the boundary of configuration space (with the associated metric) allowed us to argue easily that there could be no infinitesimal squares inscribed on a smooth curve. On the other hand, this is not the only way to address these difficulties: For instance, Stromquist~\cite{MR1045781} deals with basically the same problem by showing directly that there are no squares (or square-like quadrilaterals) smaller than some $\epsilon$ which can be inscribed on a curve with some mild smoothness assumptions and hence avoids the dangerous diagonals of the product space $(\R^k)^4$. We give a similar argument in \cite{FTCWC} to show:
\begin{theorem}[\cite{FTCWC}]
Let $\gamma\co S^1\hookrightarrow \R^n$ be an embedding of $S^1$ in $\R^n$. If $\gamma$ is in $\FTCWC$, then $\gamma$ has an inscribed square-like quadrilateral.
\label{thm:ftcwc}
\end{theorem}
We note that since this result is obtained by a limit argument, we cannot rule out the possibility that several squares come together in the limit to leave an even number of squares inscribed in the final curve, as in the examples of \cite{2008arXiv0810.4806P, MR2798015, vanheijst2014algebraic}. The appeal of this result is that it is a generalization of the square-peg problem to embedded space-curves, and that the class of curves of finite total curvature is a well-understood space~(cf.~\cite{math.GT/0606007}). When we set $n=2$, we recover the square-peg result. The regularity class of curves in \thm{ftcwc} is similar in flavor to the curves of low regularity for square-pegs given by Stromquist \cite{MR1045781}, Matschke \cite{2010arXiv1001.0186M,MR3184501}, and T.~Tao~\cite{MR3731730}, but generalizes to higher dimensions.

A very interesting possible extension of the methods here would be to use the 1-jet version of multijet transversality to try to prove a transversality theorem for submanifolds of configuration spaces which do intersect in certain boundary faces. Doing so would allow one to extend the ``counting'' and homology arguments above to detect boundary intersections between submanifolds of configuration spaces. For example, one might try to argue in this way that the space of triangles with a given angle inscribed in a curve had the homology of the torus, keeping in mind that a circle's worth of such ``triangles'' would be expected to be chords meeting the tangent to the curve in the specified angle.
Another interesting use for such a theorem would be to try to extend these theorems to immersed curves with normal crossings (as opposed to simply studying embedded curves). 

We have proved that the space of smooth curves with an odd number of squares are dense among smooth curves in the plane (or residual among smooth curves). This is not quite the same as proving that a ``generic'' smooth curve has an odd number of inscribed squares. It would be very interesting to try to extend these results to a set of curves which was full-measure among plane curves according to some natural measure on curves, as F.~Morgan does in~\cite{1978InMat..45..253M} for space curves bounding a unique area-minimizing surface.


\section*{Acknowledgments}
The authors would like to first thank Gerry Dunn who introduced us to the problem. We would also like to thank the people who have discussed the problem with us over the years: Jordan Ellenberg, Richard Jerrard, Rob Kusner, Benjamin Matschke,  Igor Pak, Strashimir Popvassiliev, John M. Sullivan, Cliff Taubes, and Gunter Ziegler.

\bibliography{ngons}{}
\bibliographystyle{plain}

\appendix

\section{Structure of the boundary of $\slq$}\label{sect:bdry-slq}
Here, we give further results about the structure of the boundary of $\slq$, which we recall lies in the  $(1234)$ and $(13)(24)$ faces of $\bdry \cfr$. \begin{proposition}
\label{prop:slq-boundary}
Each of the boundary $(1234)$ and $(13)(24)$ faces of $\slq$ is a submanifold of $\cfr$.
\end{proposition}
\begin{proof} 
Let us consider the $(1234)$ boundary face, where both the sidelengths and diagonals of the square-like quadrilateral vanish.  Following Sinha~\cite{newkey119} (Theorems 3.12 and 3.14), the boundary face  $(1234)$ is diffeomorphic to the manifold $\R^k\times \tilde{C}_4(\R^k) \times \{0\}$, where $\tilde{C}_4(\R^k)$ corresponds to configurations of 4 points up to scaling and translation. 
We are then able to smoothly extend the ratios in the definition of ${g}$ (Equation~\ref{slq-eq}) to the boundary. Let $(\bfp, \bfy_1, \bfy_2, \bfy_3,\bfy_4)\in \R^k\times \tilde{C}_4(\R^k)$, and without loss of generality assume $\bfp$ is the center of mass of the configuration, and the vectors $\sum_i\bfy_i = 4\bfp$. It is straightforward to see that $r_{ijl}^2$ becomes $\tilde{r}_{ijl}^2 = \frac{ \norm{ \bfy_i-\bfy_j}}{\norm{\bfy_i-\bfy_l} }$.  Thus the arguments in the proof of \prop{slq is submanifold of r2} carry over directly, and $g$ is transverse to $(1,1,1,0)$ on this boundary face. Hence $\slq$ is a submanifold on the $(1234)$ boundary face. 

On the $(13)(24)$ boundary face, the square-like quadrilaterals are four-fold covers of an interval with $\pi_{13}$ perpendicular to $\pi_{24}$. This information is not given by our defining map ${g}$ (Equation~\ref{slq-eq}) and we will be unable to show that $g$ is transverse on this face.  We solve this problem by finding a different map that defines $\slq$, and that is also transverse on the $(13)(24)$ face. We define the map $f:\cfr\rightarrow \R^4$ by
\begin{equation} f(\p) = \left( (\pi_{14}+\pi_{34})\cdot \pi_{13}, \ (\pi_{41}+\pi_{21})\cdot\pi_{24}, \ \pi_{13}\cdot\pi_{24}, \ r^2_{132}-r^2_{241}\right).
\end{equation}\label{eq:slq-v2}
In \prop{slq-v2} (in Appendix~\ref{sect:slq2}) we prove that $\slq = f^{-1}(0,0,0,0)$. Note that a point in $\slq$ in the $(13)(24)$ face is captured by $\p=(\bfp_1=\bfp_3, \bfp_2=\bfp_4, \pi_{13}, \pi_{24})\in\R^k\times\R^k\times S^{k-1}\times S^{k-1}$. We are only  interested in $(13)(24)$ face, so we make appropriate adjustments to $f$ and restrict it to become $\hat{f}:\cfr\rightarrow \R^3$  defined by 
$$\hat{f}(\p) = (\hat{f}_1(\p), \hat{f}_2(\p), \hat{f}_3(\p)) =  \left( (2\pi_{14})\cdot \pi_{13},\  (2\pi_{41})\cdot\pi_{24},\  \pi_{13}\cdot\pi_{24}\right).$$ 
By design, $\slq=\hat{f}^{-1}(0,0,0)$ on the $(13)(24)$ face.  In order to use the Preimage Theorem of \cite{Guillemin:2010ti}, we follow the ideas behind the proof of \prop{slq is submanifold of r2}, and find three tangent vectors to $\cfr$ on which it is easy to show $D\hat{f}$ has three linearly independent rows.  We first compute a typical column of $D\hat{f}$, where we have differentiated with respect to a vector $\cv$.
\begin{equation}\label{eq:Df}
\begin{pmatrix} D_{\cv}(2\pi_{14})\cdot \pi_{13} + (2\pi_{14})\cdot D_{\cv}\pi_{13}
\\  D_{\cv}(2\pi_{41})\cdot \pi_{24} + (2\pi_{41})\cdot D_{\cv}\pi_{24}
\\ D_{\cv}(\pi_{13})\cdot \pi_{24} + (\pi_{13})\cdot D_{\cv}\pi_{24}
\end{pmatrix}
\end{equation}

\begin{center}
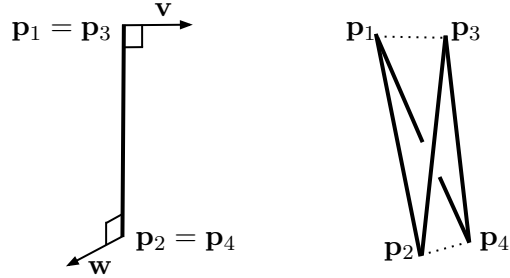
\begin{figure}[htbp]
\hfill
\begin{overpic}{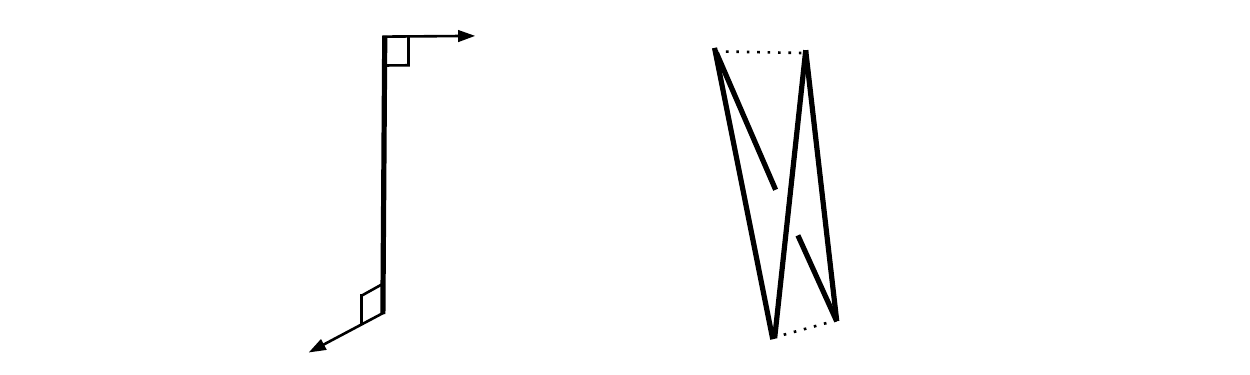}
\put(19,26){$\bfp_1=\bfp_3$}
\put(34,28){$\bfv$}
\put(32,4){$\bfp_2=\bfp_4$}
\put(27,1){$\bfw$}
\put(54, 26){$\bfp_1$}
\put(58,3){$\bfp_2$}
\put(65,26){$\bfp_3$}
\put(68,4){$\bfp_4$}
\end{overpic}
\caption{On the left a point $\p$ in the $(13)(24)$ face of $\bdry\slq$. The vectors $\bfv$ and $\bfw$ are the variations which move $\p$ to a square-like quadrilateral in the interior of $\slq$. This is shown on the right. The vectors $\bfv$ and $\bfw$ are perpendicular to each other, and also to the line (of symmetry) through $\bfp_1=\bfp_3$ and $\bfp_2=\bfp_4$.
}
\label{fig:slq-boundary}
\end{figure}
\end{center}

The definition of $\slq$ on the $(13)(24)$ face guarantees that the line segment $\bfp_1\bfp_4$ is perpendicular to both $\pi_{13}$ and $\pi_{24}$.  The tangent vectors to $\slq$ on the $(13)(24)$ boundary face may be represented by a vector $\cv=(\bfv_1,\bfv_2,\bfv_3,\bfv_4)$, where $\bfv_1,\bfv_2\in\R^k$ are tangent vectors to $\bfp_1$ and $\bfp_4$ respectively, and $\bfv_3,\bfv_4\in S^{k-1}$ are tangent vectors to $\pi_{13}$ and $\pi_{24}$ respectively.  
Our first tangent vector is $\cv_1=(\bfzero,\bfzero,\frac{1}{2}\pi_{14}, \bfzero)$, which moves $\pi_{13}$ in the $\pi_{14}$ direction (perpendicular to both $\pi_{13}$ and $\pi_{24}$). Thus $D_{\cv_1}\pi_{13}=\frac{1}{2}\pi_{14}$ and the directional derivatives of the other $\pi_{ij}$ are 0. A short computation using Equation~\ref{eq:Df} shows $D_{\cv_1}\hat{f}_1 = \bfzero\cdot\pi_{13} + (2\pi_{14})\cdot(\frac{1}{2}\pi_{14} )= 1$, $D_{\cv_1}\hat{f}_2=0$, and $D_{\cv_1}\hat{f}_3 = (\frac{1}{2}\pi_{14})\cdot\pi_{24} + \pi_{13}\cdot\bfzero = 0$. We repeat this computation for our second tangent vector $\cv_2=(\bfzero,\bfzero,\bfzero,\frac{1}{2}\pi_{41})$  which moves $\pi_{24}$ in the $\pi_{41}$ direction (perpendicular to both $\pi_{13}$ and $\pi_{24}$), 
and also for our third tangent vector $\cv_3=(\bfzero, \bfzero, \pi_{24}, \bfzero)$ which moves $\pi_{13}$ in the $\pi_{24}$ direction (perpendicular to both $\pi_{13}$ and $\pi_{41}$).  With respect to the basis vectors $\pi_{13}, \pi_{24}, \pi_{41}$, we find
$$D\hat{f} = \begin{pmatrix} 1&0&0 \\ 0 & 1 & 0 \\ 0 & 0 & 1
\end{pmatrix}.
$$
Hence $D\hat{f}$ is onto and $\hat{f}$ is transverse to $(0,0,0)$,  and thus $\slq$ is a submanifold on the boundary face $(13)(24)$.
\end{proof}

\begin{remark} The reader might wonder why we did not simply use $f$ to define $\slq$ throughout the paper. It turns out that it is much, much harder to find appropriate tangent vectors to use in the proofs of Propositions~\ref{prop:slq is submanifold of r2}, \ref{prop:slq-boundary}, and \ref{prop:orientation}. We chose to simplify the computations at the cost of using two functions to define $\slq$.
\end{remark}

We conclude this section by noting that we have not proven the stronger result that  $\slq$ is a submanifold-with-boundary and corners of $C_4[R^k]$. While this is most likely true, we do not need this, or indeed any of the results in Appendices~\ref{sect:bdry-slq} or \ref{sect:slq2}, for our main theorem.

\section{A second approach to square-like quadrilaterals.}\label{sect:slq2}
The results in this appendix  fill in the details needed for the proof of Proposition~\ref{prop:slq-boundary}.
Recall that we define $\slq$ to  be the subset of {\em square-like quadrilaterals} of $\cfr$ such that  $r_{124} = r_{231} = r_{342} = 1$ and $r_{132} - r_{241} = 0$.  Equivalently, a point $\p=(\bfp_1,\bfp_2,\bfp_3,\bfp_4,\alpha(\bfp))\in\cfr$ is a square-like quadrilateral when it has equal sides and equal diagonals. That is
\begin{align} \norm{\bfp_1-\bfp_2}=\norm{\bfp_2-\bfp_3} &= \norm{\bfp_3-\bfp_4} = \norm{\bfp_4-\bfp_1}\label{eq:equal-sides}
\\
\norm{\bfp_1-\bfp_3} & = \norm{\bfp_2-\bfp_4}\label{eq:equal-diagonals}
\end{align}

Before we continue we recall a result about dot product. Assume that for four unit vectors $\vec{u}_1, \vec{u}_2, \vec{v}_1, \vec{v}_2$,  we have $\vec{u}_1 \cdot \vec{v}_1 = \vec{u}_2\cdot\vec{v}_2$. Then $\cos\theta_1=\cos\theta_2$, where $0\leq \theta_i\leq \pi$ are the angles between the vectors. Hence $\theta_1=\theta_2$. The converse also holds.

Again, assume that $\bfp_1\bfp_2\bfp_3\bfp_4$ is a square-like quadrilateral, so that $\triangle \bfp_1\bfp_4\bfp_2$ is isosceles. Then two internal angles ($\angle \bfp_1\bfp_4\bfp_2 = \angle \bfp_1\bfp_2\bfp_4$) of the triangle are equal, and their corresponding external angles are also equal. In addition, the line joining $\bfp_1$ to the midpoint of $\bfp_2\bfp_4$ bisects $\angle \bfp_2\bfp_1\bfp_4$ and is an altitude. Translating these ideas to taking the dot product of unit vectors gives the following equivalent equations:
\begin{align*}
\pi_{14}\cdot\pi_{24} = \pi_{12}\cdot\pi_{42} & = -\pi_{12}\cdot\pi_{24}\qquad\text{equal internal angles,}
\\(\pi_{14} + \pi_{12})\cdot\pi_{24} & =0
\\ -(\pi_{14} + \pi_{12})\cdot\pi_{24} & =0 \qquad\text{angle bisector is the altitude,}
\\ (\pi_{41} + \pi_{21})\cdot\pi_{24} & =0 \qquad\text{we use this equation below,}
\\ \pi_{41}\cdot\pi_{24} &= \pi_{21}\cdot\pi_{42} \qquad\text{equal external angles.} 
\end{align*}
Any of these five equations implies that $\triangle \bfp_1\bfp_2\bfp_4$ is isosceles and that $\norm{\bfp_1-\bfp_4}= \norm{\bfp_1-\bfp_2}$.

We now prove there is a second way of defining $\slq$ (first seen in Equation~\ref{eq:slq-v2}). Recall that we defined the map $f:\cfr\rightarrow \R^4$  by
\begin{equation*} f(\p) = \left( (\pi_{14}+\pi_{34})\cdot \pi_{13}, \ (\pi_{41}+\pi_{21})\cdot\pi_{24},  \ \pi_{13}\cdot\pi_{24},\  r^2_{132}-r^2_{241}\right).
\end{equation*}

\begin{proposition} \label{prop:slq-v2} The quadrilateral $\bfp_1\bfp_2\bfp_3\bfp_4$ is square-like (satisfies Equations~\ref{eq:equal-sides} and~\ref{eq:equal-diagonals}) if and only if $\p=(\bfp_1,\bfp_2,\bfp_3,\bfp_4,\alpha(\bfp))\in f^{-1}(0,0,0,0)$.
\end{proposition}

\begin{proof}
First assume that $\bfp_1\bfp_2\bfp_3\bfp_4$ has equal sides and equal diagonals. Then triangles $\triangle \bfp_4\bfp_1\bfp_3$ and $\triangle \bfp_1\bfp_2\bfp_4$ are both isosceles triangles. The discussion above shows that this implies $(\pi_{14}+\pi_{34})\cdot \pi_{13}=0$, and $(\pi_{41}+\pi_{21})\cdot\pi_{24}=0$. Since the diagonals are equal in length,  $r^2_{132}-r^2_{241}=0$ is automatically true.   

Note that showing $\pi_{13}\cdot\pi_{24}=0$, is the same as showing the diagonals $\bfp_1\bfp_3$ and $\bfp_2\bfp_4$ are perpendicular to one another. 
When $k=2$ (or the quadrilateral is planar), then the square-like quadrilateral is in fact a square, and the diagonals of squares are perpendicular. When $k>2$, we need a different argument. Set $\bfm_1$ to be the midpoint of $\bfp_1\bfp_3$ and $\bfm_2$ to be the midpoint of $\bfp_2\bfp_4$. Then $\triangle \bfp_1\bfp_2\bfp_4$ has altitude $\bfp_1\bfm_2$, and $\triangle \bfp_3\bfp_2\bfp_4$ has altitude $\bfp_3\bfm_2$. Thus the plane through the points $\bfp_1\bfp_3\bfm_2$ is perpendicular to $\bfp_2\bfp_4$, and so $\pi_{13}\cdot\pi_{24}=0$.

Now assume that $\p=(\bfp_1,\bfp_2,\bfp_3,\bfp_4,\alpha(\bfp))\in f^{-1}(0,0,0,0)$. Since $r^2_{132}-r^2_{241}=0$, then the diagonals are equal in length. Now $(\pi_{14}+\pi_{34})\cdot \pi_{13}=0$ and  $(\pi_{41}+\pi_{21})\cdot\pi_{24}=0$, shows that  $\triangle \bfp_4\bfp_1\bfp_3$ and $\triangle \bfp_1\bfp_2\bfp_4$ are both isosceles triangles. Since these triangles share side $\bfp_1\bfp_4$, we get $\norm{\bfp_3-\bfp_4} = \norm{\bfp_4-\bfp_1} = \norm{\bfp_1-\bfp_2}$.  Using the same notation as above, we let $\bfm_2$ be a point on $\bfp_2\bfp_4$ such that $\bfp_1\bfm_2$ is an altitude of $\triangle \bfp_1\bfp_2\bfp_4$. Thus $\bfp_1\bfm_2$ is perpendicular to $\bfp_2\bfp_4$. Since $\pi_{13}\cdot\pi_{24}=0$, then $\bfp_1\bfp_3$ is perpendicular to $\bfp_2\bfp_4$. This means that the plane though $\bfp_1\bfp_3\bfm_2$ is perpendicular to $\bfp_2\bfp_4$, and so
$\bfp_3\bfm_2$ is also perpendicular to $\bfp_2\bfp_4$. Thus $\triangle \bfp_4\bfm_2\bfp_3\cong\triangle \bfp_2\bfm_2\bfp_3$ (SAS) and hence $\norm{\bfp_4-\bfp_3}=\norm{\bfp_2-\bfp_3}$. Altogether we see that all the sides have the same length, and so the quadrilateral is indeed square-like.
\end{proof}

%

\section{The ellipse}\label{sect:ellipse}
In order to complete the proof of our main theorem, we need to show that in any ellipse, there is a single inscribed square.

\begin{lemma}\label{lem:one-square-ellipse}
In $\R^2$, if $\gamma$ is a planar ellipse $ \nicefrac{x^2}{a^2} + \nicefrac{y^2}{b^2} = 1$ with $a^2 \neq b^2$, then $\qcfog \cap \qslq\neq\emptyset$ and the intersection represents a single square.
\end{lemma}

\begin{proof} 
We will need a lemma:

\begin{lemma}\label{lem:parallelchords}
Parallel chords meeting an ellipse have midpoints on a line through the center of the ellipse (where the major and minor axes meet).
\end{lemma}
\begin{proof} 
This is true for a circle and is preserved under affine mappings.
\end{proof}

We prove that the intersection $\qcfog \cap \qslq$ is a single square. First, if we intersect the ellipse with the lines $y = \pm x$, by symmetry the intersection points form a square.  We prove that this is the only square inscribed in the ellipse.
If we parametrize the ellipse by $(x(\theta),y(\theta)) = (a \cos \theta, b \sin\theta)$, we can work out that $\cos^2 \theta = b^2/(a^2 + b^2)$ and $\sin^2\theta = a^2/(a^2 + b^2)$. 

Suppose $ABCD$ is any square inscribed in the ellipse. Let $M$ denote the midpoint of $AB$
and~$N$ denote the midpoint of $CD$. Then, by \lem{parallelchords}, $MN$ passes
through the center $O$ of the ellipse. Similarly, if $K$ denotes the midpoint of $AD$ and
$L$ the midpoint of $BC$, then $KL$ passes through~$O$. Thus $O$ is also the center of the square.
Using our parametrization of the ellipse, we can write
\begin{equation*}
A = (a\cos\alpha, b\sin\alpha), \quad B = (a\cos\beta, b\sin \beta).
\end{equation*}
The segment $OM$ is perpendicular to $AB$ and so $\triangle OAM$ and $\triangle OBM$
are congruent and $OA\cong OB$. Thus
\begin{equation*}
a^2 \cos^2\alpha + b^2 \sin^2 \alpha = a^2\cos^2 \beta + b^2\sin^2\beta.
\end{equation*}
This implies $(a^2 - b^2) \cos^2 \alpha = (a^2 - b^2)\cos^2 \beta$ and so, since $a\neq b$, we know $\cos \alpha
= \pm \cos\beta$. Similarly, $\sin \alpha = \pm \sin \beta$. This means that $B$ is the image of $A$ under a symmetry of the ellipse, and since the same argument works \textit{mutatis mutandis} for $C$ and $D$, the square is symmetric under the flip symmetries of the ellipse. There are two types of inscribed quadrilaterals with this symmetry: inscribed rectangles in the form $(\pm x, \pm y)$, and the ``exceptional'' rhombus $\{(\pm a,0),(0,\pm b)\}$. Since $a \neq b$, the only square is our previous set of 4 points $(\pm ab/\sqrt{a^2 + b^2}, \pm ab/\sqrt{a^2+b^2})$.
\end{proof}


\end{document}